\documentclass[11pt,a4paper]{amsart}
\pdfoutput=1
\usepackage{graphicx}
\usepackage{dsfont}
\usepackage[lite,initials,msc-links]{amsrefs}

\newtheorem{theorem}{Theorem}[section]
\newtheorem{corollary}[theorem]{Corollary}
\newtheorem{proposition}[theorem]{Proposition}
\newtheorem{lemma}[theorem]{Lemma}

\theoremstyle{definition}

\newtheorem{remark}[theorem]{Remark}

\numberwithin{equation}{section}

\newcommand{\supf}{0.515428}
\newcommand{\inff}{0.477449}

\let\P\relax											% undefines \P
\DeclareMathOperator{\P}{P}								% probability
\DeclareMathOperator{\E}{E}								% expectation
\DeclareMathOperator{\I}{\mathds{1}}					% indicator

\newcommand{\delims}[4]{\mathopen#1#2#4\mathclose#1#3}
\newcommand{\abs}[2][]{\delims{#1}\lvert\rvert{#2}}		% absolute value
\newcommand{\card}[2][]{\delims{#1}\lvert\rvert{#2}}	% cardinality
\newcommand{\floor}[2][]{\delims{#1}\lfloor\rfloor{#2}}	% floor
\newcommand{\ceil}[2][]{\delims{#1}\lceil\rceil{#2}}	% ceiling

\newcommand{\N}{\mathds{N}}								% natural numbers
\newcommand{\R}{\mathds{R}}								% real numbers
\newcommand{\Z}{\mathds{Z}}								% integers

\newcommand{\Mathematica}{\textit{Mathematica}}			% Mathematica

\begin{document}

\title{The asymptotics of group Russian roulette}
\author{Tim van de Brug}
\author{Wouter Kager}
\author{Ronald Meester}
\date{\today}
\address{Vrije Universiteit\\
Faculty of Sciences\\
Department of Mathematics\\
De Boelelaan 1081a\\
1081\,HV Amsterdam\\
The Netherlands}
\email{\char`{t.vande.brug,w.kager,r.w.j.meester\char`}\,@\,vu.nl}

\begin{abstract}
	We study the group Russian roulette problem, also known as the shooting 
	problem, defined as follows. We have $n$~armed people in a room. At each 
	chime of a clock, everyone shoots a random other person. The persons shot 
	fall dead and the survivors shoot again at the next chime. Eventually, 
	either everyone is dead or there is a single survivor. We prove that the 
	probability~$p_n$ of having no survivors does not converge as $n\to 
	\infty$, and becomes asymptotically periodic and continuous on the 
	$\log{n}$~scale, with period~1.
\end{abstract}


\maketitle

\section{Introduction and main result}
\label{sec:introduction}

In~\cite{Winkler}, Peter Winkler describes the following probability puzzle, 
called group Russian roulette, and also known as the shooting problem. We 
start at time $t=0$ with $n$~people in a room, all carrying a gun. At time 
$t=1$, all people in the room shoot a randomly chosen person in the room; it 
is possible that two people shoot each other, but no one can shoot him- or 
herself. We assume that every shot instantly kills the person shot at. After 
this first shooting round, a random number of people have survived, and at 
time $t=2$ we repeat the procedure with all survivors. Continuing like this, 
eventually we will reach a state with either no survivors, or exactly one 
survivor. Denote by~$p_n$ the probability that eventually there are no 
survivors. We are interested in the behavior of~$p_n$ as $n\to \infty$.

Observe that the probability that a given person survives the first shooting 
round is $( 1-(n-1)^{-1} )^{n-1} \approx 1/e$, so that the expected number of 
survivors of the first round is approximately~$n/e$. This fact motivates us to 
plot $p_n$ against~$\log{n}$, see Figure~\ref{figure1} below. 
Figure~\ref{figure1} suggests that $p_n$ does not converge as $n\to \infty$, 
and becomes asymptotically periodic on the~$\log{n}$ scale, with period~1. 
This turns out to be correct, and is perhaps surprising. One may have 
anticipated that, as $n$ gets very large, the fluctuations at every round will 
somehow make the process forget its starting point, but this is not the case. 
Indeed, here we prove the following:

\begin{theorem}
	\label{thm:main}
	There exists a continuous, periodic function $f\colon \R\to [0,1]$ of 
	period~1, satisfying $\sup f\geq \supf$ and $\inf f\leq \inff$, such 
	that
	\[
		\sup_{x\geq x_0} \, \abs[\big]{ p_{\floor{\exp x}} - f(x) } \to 0
		\qquad \text{as $x_0\to\infty$}.
	\]
\end{theorem}

\begin{figure}
	\begin{center}
		\includegraphics[width=\textwidth]{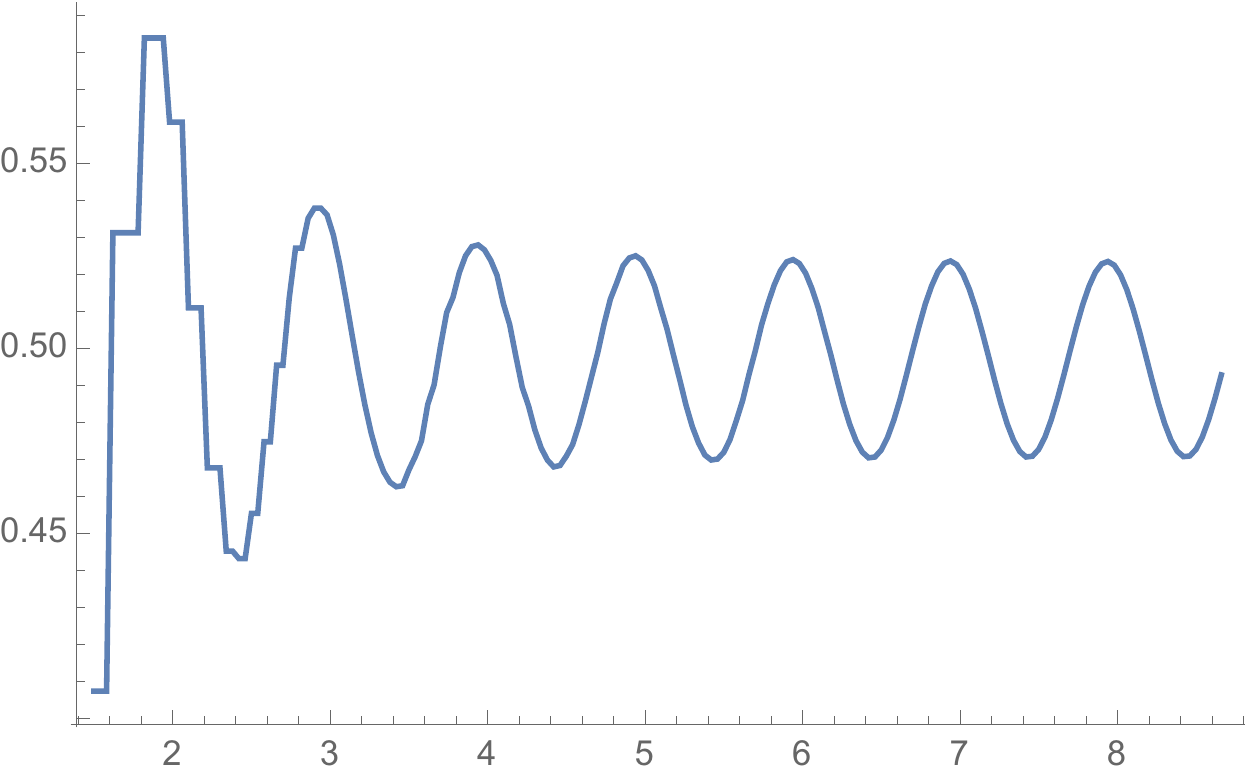}
	\end{center}
	\caption{$p_n$ as a function of~$\log{n}$ up to $n=6000$.}
	\label{figure1}
\end{figure}

The solution to the group Russian roulette problem as it is stated in 
Theorem~\ref{thm:main} was already stated in~\cite{Winkler}, without the 
explicit bounds on the limit function. However, \cite{Winkler} does not 
provide a proof, and as far as we know, there is no proof in the literature.

A number of papers \cites{LiPom, AthFid, RadGriLos, EisSteStr, Pro, BruOci, 
Kin, BraSteWil} study the following related problem and generalizations 
thereof. Suppose we have $n$~coins, each of which lands heads up with 
probability~$p$. Flip all the coins independently and throw out the coins that 
show heads. Repeat the procedure with the remaining coins until 0 or~1 coins 
are left. The probability of ending with 0~coins does not converge as $n\to 
\infty$ and becomes asymptotically periodic and continuous on the $\log n$ 
scale \cites{EisSteStr,RadGriLos}. For $p=1- 1/e$, the limit function takes 
values between 0.365879 and 0.369880, see~\cite{EisSteStr}*{Corollary~2}.

The coin tossing problem for $p=1-1/e$ has some similarities with group 
Russian roulette. In view of Theorem~\ref{thm:main} and the results in 
\cites{EisSteStr,RadGriLos}, the asymptotic behavior of these two models is 
qualitatively similar but their limit functions have different average values 
and amplitudes. In the above-mentioned papers, explicit expressions for the 
probability of ending with no coins could be obtained because of the 
independence between coin tosses. Analytic methods were subsequently employed 
to evaluate the limit. This strategy does not seem applicable to the group 
Russian roulette problem for the simple reason that no closed-form expressions 
can be obtained for the relevant probabilities. Our approach is, therefore, 
very different, and we end this introduction with an overview of our strategy.

We recursively compute rigorous upper and lower bounds on~$p_n$ for 
$n=1,\dots,6000$, using \Mathematica. Based on these computations, we identify 
values of~$n$ where $p_n$ is high (the ``hills'') and values of~$n$ where 
$p_n$ is low (the ``valleys''). To prove the non-convergence of the~$p_n$, we 
explicitly construct intervals $H_k$ and~$V_k$ ($k=0,1,\dotsc$) in such a way 
that, if $n\in H_k$ for some~$k$, then with high probability uniformly in~$k$, 
the number of survivors in the shooting process starting with $n$~people will, 
during the first~$k$ shooting rounds, visit each of the intervals $H_{k-1}, 
H_{k-2}, \dots, H_0$ (in that order), and similarly for the~$V_k$. By our 
rigorous bounds on~$p_n$ we know that $H_0$ is a hill and~$V_0$ a valley. This 
implies that the values of~$p_n$ on the respective intervals $H_k$ and~$V_k$ 
are separated from each other, uniformly in~$k$.

We stress that, although we make use of \Mathematica, our proof of 
Theorem~\ref{thm:main} is completely rigorous. There are no computer 
simulation methods involved, and we use only integer calculations to avoid 
rounding errors. To make this point clear, we isolated the part of the proof 
where we use \Mathematica\ as a separate lemma, Lemma~\ref{lem:mathematica}. 
In the proof of this lemma we explain how we compute the rigorous bounds we 
need. Our \Mathematica\ notebook and bounds on the~$p_n$ are available online 
at \url{http://arxiv.org/format/1507.03805}.

A generic bound on the probability that the number of survivors after each 
round successively visits the intervals in a carefully constructed sequence, 
appears in Section~\ref{sec:intervals} below. To obtain a good bound on this 
probability, we make crucial use of a coupling, introduced in 
Section~\ref{sec:coupling}, which allows us to compare the random number of 
survivors of a single shooting round with the number of empty boxes remaining 
after randomly throwing balls into boxes. For this latter random variable 
reasonably good tail bounds are readily available, and we provide such a bound 
in Section~\ref{sec:tails}.

The coupling is also crucial in proving the asymptotic continuity and 
periodicity of the~$p_n$ on the $\log n$ scale. To prove continuity, we 
consider what happens if we start the shooting process from two different 
points in the same interval, using for every round an independent copy of the 
coupled numbers of survivors for each point. By carefully analyzing the 
properties of our coupling, we will show that we can make the two coupled 
processes collide with arbitrarily high probability before reaching 0 or~1, by 
making the intervals sufficiently narrow on the $\log n$ scale, and taking the 
interval we start from far enough to the right. This shows that for our two 
starting points, the probabilities of eventually having no survivors must be 
very close to each other. Periodicity follows because our argument also 
applies when we start from two points that lie in different intervals, and the 
distance between the intervals in our construction is~1 on the $\log n$ scale.

The proof of non-convergence of the~$p_n$, based on the coupling and tail 
bounds from Section~\ref{sec:prelims}, is in Section~\ref{sec:nonconvergence}. 
The proof of asymptotic periodicity and continuity follows in 
Section~\ref{sec:periodicity}. Together, these results give 
Theorem~\ref{thm:main}.

\section{Coupling and tail bounds}
\label{sec:prelims}

\subsection{Coupling and comparison}
\label{sec:coupling}

Let $S_n$ be the number of survivors after one round of the shooting process 
starting with $n$~people. Using inclusion-exclusion, the distribution of~$S_n$ 
can be written down explicitly:
\begin{multline}
	\label{Sn-incl-excl}
	\P(S_n=k)
	= \binom{n}{k} (n-1)^{-n} \times \\
	  \sum_{r=0}^{n-k-2}\binom{n-k}{r} (-1)^r(n-k-r)^{k+r} (n-k-r-1)^{n-k-r}.
\end{multline}
We use this formula in Section~\ref{sec:nonconvergence}, but not in the rest 
of our analysis. Instead, let $Y_n$ be a random variable that counts the 
number of boxes that remain empty after randomly throwing $n-1$ balls into 
$n-1$ (initially empty) boxes. Similarly, let~$Z_n$ be the result of adding~1 
to the number of boxes that remain empty after randomly throwing $n$~balls 
into $n-1$ boxes. It turns out that these random variables $Y_n$ and~$Z_n$ are 
very close in distribution to~$S_n$, and are more convenient to work with.

In this section we describe a coupling between the $S_n$, $Y_n$ and~$Z_n$, for 
all $n\geq2$ simultaneously, in which (almost surely) $S_n$, $Y_n$ and~$Z_n$ 
are within distance~1 from each other for all~$n$, and the $Y_n$ and~$Z_n$ are 
ordered in~$n$ (see Lemma~\ref{lem:coupling} below). This last fact has the 
useful implication that in the shooting problem, if the number~$n$ of people 
alive in the room is known to be in an interval $[a, b]$, then the probability 
that the number of survivors of the next shooting round will lie in some other 
interval $[\alpha, \beta]$ can be estimated by considering only the two 
extreme cases $n=a$ and~$n=b$ (see Corollary~\ref{cor:coupling} below). At the 
end of the section, we extend our coupling to a coupling we can use to study 
shooting processes with multiple shooting rounds.

To describe our coupling, we construct a Markov chain as follows. Number the 
people $1, 2, \dots, n$ and define $A^n_i \subset \{1,\dots,n\}$ as the set of 
people who are not shot by any of the persons~1 up to~$i$ (inclusive). In this 
formulation, $S^n_i := \card{A^n_i}$ represents the number of survivors if 
only persons~1 up to~$i$ shoot, and we can write
\[
	S_n = S_n^n = \card{A^n_n}.
\]
The sets~$A^n_i$ ($i=1,\dots,n$) form a Markov chain inducing the 
process~$(S^n_i)_i$ with transition probabilities given by
\begin{equation}
\begin{split}
	\P(S^n_{i+1} = S^n_i - 1 \mid A^n_i)
	&= 1- \P(S^n_{i+1} = S^n_i \mid A^n_i) \\
	&= \frac{S^n_i - \I(i+1\in A^n_i)}{n-1}.
	\label{Mprocess1}
\end{split}
\end{equation}
Indeed, when person~$i+1$ selects his target, the number of persons who will 
survive the shooting round decreases by~1 precisely when person~$i+1$ aims at 
someone who has not already been targeted by any of the persons~1 up to~$i$, 
where we must take into account that person~$i+1$ cannot shoot himself (hence 
the subtraction of $\I(i+1\in A^n_i)$ in the numerator).

An explicit construction of the process described above can be given as 
follows. Suppose that on some probability space, we have random variables 
$U_1, U_2, \dotsc$ uniformly distributed on $(0,1]$, and, for all finite 
subsets~$A$ of~$\N$ and all $i\in\N$, random variables~$V_{A,i}$ uniformly 
distributed on the set $A\setminus\{i\}$, all independent of each other. Now 
fix $n\geq2$. Set $S^n_0 := n$ and $A^n_0 := \{1,\dots,n\}$, and for 
$i=0,1,\dots,n-1$, recursively define
\[
	A^n_{i+1} := \begin{cases}
		A^n_i \setminus \{V_{A^n_i,i+1}\}
		& \text{if } {\displaystyle U_{n-i} \leq \frac{\card{A^n_i} - 
		\I(i+1\in A^n_i)}{n-1}}; \\
		A^n_i
		& \text{otherwise};
	\end{cases}
\]
and set $S^n_{i+1} := \card{A^n_{i+1}}$. In this construction, the variable 
$U_{n-i}$ is used first to decide whether person~$i+1$ aims at someone who 
will not be shot by any of the persons~1 up to~$i$, and then we 
use~$V_{A^n_i,i+1}$ to determine his victim. Clearly, this yields a process 
with the desired distribution, and provides a coupling of the 
processes~$(S^n_i)_i$ for all $n\geq2$ simultaneously.

We now extend this coupling to include new processes $(Y^n_i)_i$ 
and~$(Z^n_i)_i$, as follows. For fixed $n\geq2$, we first set $Y^n_0 := n$ and 
$Z^n_0 := n$, and then for $i=0,1,\dotsc,n-1$ we recursively define
\begin{align}
	\label{Yprocess}
	Y^n_{i+1} &:= \begin{cases}
		Y^n_i-1 & \text{if }{\displaystyle U_{n-i}\leq\frac{Y^n_i}{n-1}};\\
		Y^n_i   & \text{otherwise};
	\end{cases} \\
\intertext{and}
	\label{Zprocess}
	Z^n_{i+1} &:= \begin{cases}
		Z^n_i-1 & \text{if }{\displaystyle U_{n-i}\leq\frac{Z^n_i-1}{n-1}};\\
		Z^n_i   & \text{otherwise}.
	\end{cases}
\end{align}
Then, by construction, $(Y^n_i)_i$ and $(Z^n_i)_i$ are Markov chains with the 
respective transition probabilities
\begin{align}
	\P(Y^n_{i+1} = Y^n_i-1 \mid Y^n_i)
	&= 1-\P(Y^n_{i+1} = Y^n_i \mid Y^n_i) = \frac{Y^n_i}{n-1};
	 \label{Mprocess2} \\
	\P(Z^n_{i+1} = Z^n_i-1 \mid Z^n_i)
	&= 1-\P(Z^n_{i+1} = Z^n_i \mid Z^n_i) = \frac{Z^n_i-1}{n-1}.
	 \label{Mprocess3}
\end{align}

The similarity with~\eqref{Mprocess1} is clear, and we see that we can 
interpret $Y^n_i$ as the number of empty boxes after throwing $i$~balls into 
$n$~boxes, where the first ball is thrown into the $n$th box and the remaining 
balls are thrown randomly into the first $n-1$ boxes only. Likewise, $Z^n_i$ 
is the number of empty boxes after throwing $i$~balls into the first $n-1$ of 
a total of $n$~boxes (so that the $n$th box remains empty throughout the 
process). If we now set
\[
	S_n := S^n_n, \quad Y_n := Y^n_n \text{ and } Z_n := Z^n_n,
\]
then $S_n$, $Y_n$ and~$Z_n$ have the interpretations described at the 
beginning of this section. The next lemma shows they have the properties we 
mentioned:

\begin{lemma}
	\label{lem:coupling}
	The coupling of the $S_n$, $Y_n$ and~$Z_n$ described above satisfies
	\begin{enumerate}
		\item $Y_n\leq Y_{n+1}\leq Y_n+1$ and $Z_n\leq Z_{n+1} \leq Z_n+1$ for 
			all $n\geq2$;
		\item $Y_n\leq S_n\leq Z_n\leq Y_n+1$ for all $n\geq2$.
	\end{enumerate}
\end{lemma}

\begin{proof}
	As for~(1), we claim that the~$Y^n_i$ satisfy the stronger statement that
	\begin{equation}
		Y^n_i \leq Y^{n+1}_{i+1} \leq Y^n_i+1
		\qquad \text{for all $n\geq2$ and $i=0,1,\dots,n$}.
		\label{ordering1}
	\end{equation}
	To see this, first note that necessarily, $Y^{n+1}_1 = n = Y^n_0$. Now 
	suppose that $Y^n_i = Y^{n+1}_{i+1}$ for some index~$i$. Then 
	\eqref{Yprocess} implies that if $Y^{n+1}_{i+2} = Y^{n+1}_{i+1} - 1$, we 
	also have $Y^n_{i+1} = Y^n_i - 1$. Hence the ordering is preserved, 
	proving that $Y^n_i \leq Y^{n+1}_{i+1}$ for all $i\leq n$. Likewise, if 
	$Y^{n+1}_{i+1} = Y^n_i+1$ and $Y^n_{i+1} = Y^n_i - 1$ for some~$i$, then 
	\eqref{Yprocess} implies that $Y^{n+1}_{i+2} = Y^{n+1}_{i+1} - 1$. This 
	proves~\eqref{ordering1} and hence~(1) for the~$Y_n$. The proof for the 
	random variables~$Z_n$ is similar.

	As for property~(2), observe that if $Y^n_i = Z^n_i$ for some index~$i$ 
	and $Z^n_{i+1} = Z^n_i - 1$, then also $Y^n_{i+1} = Y^n_i - 1$. On the 
	other hand, if $Y^n_i = Z^n_i - 1$ for some index~$i$, then it follows 
	from the construction that $Y^n_j = Z^n_j - 1$ for all $j=i,i+1,\dots,n$. 
	Since $Y^n_0 = Z^n_0 = n$, we conclude that
	\begin{equation}
		Y^n_i \leq Z^n_i \leq Y^n_i+1
		\qquad \text{for all $n\geq2$ and $i=0,1,\dots,n$}.
		\label{ordering2}
	\end{equation}
	Furthermore, if $Y^n_i = S^n_i$ and $S^n_{i+1} = S^n_i - 1$, then our 
	construction implies that $Y^n_{i+1} = Y^n_i - 1$. Similarly, if $S^n_i = 
	Z^n_i$ and $Z^n_{i+1} = Z^n_i - 1$, then in our coupling we also have that 
	$S^n_{i+1} = S^n_i - 1$. It follows that
	\[
		Y^n_i \leq S^n_i \leq Z^n_i
		\qquad \text{for all $n\geq2$ and $i=0,1,\dots,n$},
	\]
	and this together with~\eqref{ordering2} establish property~(2).
\end{proof}

\begin{corollary}
	\label{cor:coupling}
	Suppose we have coupled the~$S_n$ as described above. Then, for any 
	intervals $[a,b]$ and $[\alpha,\beta]$, with $a,b,\alpha,\beta$ integers,
	\[
		\P(\exists n\in [a,b]\colon S_n\notin [\alpha,\beta])
		\leq \P(Y_a \leq \alpha-1) + \P(Y_b \geq \beta).
	\]
\end{corollary}

\begin{proof}
	Let the $S_n$ and~$Y_n$ be coupled as described above. By 
	Lemma~\ref{lem:coupling},
	\[
	\begin{split}
		\P(\forall n \in [a,b]\colon S_n\in [\alpha,\beta])
		&\geq \P(\forall n \in [a,b]\colon Y_n \in [\alpha,\beta -1 ]) \\
		&=    \P(Y_a \geq \alpha, Y_b \leq \beta -1).
	\end{split}
	\]
	By taking complements the desired result follows.
\end{proof}

\begin{remark}
	The distribution of $Y_i^n$ is related to Stirling numbers of the second 
	kind, as follows. Recall that $Y_{i+1}^{n+1}$ can be interpreted as the 
	number of empty boxes after throwing $i$~balls randomly into $n$~boxes. We 
	claim that
	\[\begin{split}
		\P(Y_{i+1}^{n+1} = n-k)
		&= \P(\mbox{$n-k$ boxes empty, $k$ boxes non-empty}) \\
		&= \frac{n!}{(n-k)!} \frac1{n^i} S(i,k),
	\end{split}\]
	with $S(i,k)$ a Stirling number of the second kind. Indeed, $S(i,k)$ is by 
	definition the number of ways of partitioning the set of $i$~balls into 
	$k$ non-empty subsets. Balls in the same subset are thrown into the same 
	box. The number of ways to assign these subsets to $k$ distinct boxes 
	equals $n!/(n-k)!$. Finally, $n^i$ is the number of ways of distributing 
	$i$~balls over $n$~boxes.
\end{remark}

We now extend our coupling to a coupling we can use for an arbitrary number of 
shooting rounds, and for shooting processes starting from different values 
of~$n$. Since the shooting rounds must be independent, we take an infinite 
number of independent copies of the coupling described above, one for each 
element of~$\Z$ (so including the negative integers). The idea is to use a 
different copy for each round of a shooting process. For reasons that will 
become clear, we want to allow the copy that is used for the first round to 
vary with the starting point~$n$.

To be precise, let $X^n_i$ represent the number of survivors after round~$i$ 
of a shooting process started with~$n$ people in the room. Let~$k_n$ be the 
number of the copy of our coupling that is to be used for the first round of 
this process, and denote the $i$-th copy of~$S_n$ by~$S_n^{(i)}$. We 
recursively define
\begin{equation}
	\label{Xni}
	X^n_0 := n, \qquad
	X^n_{i+1} := S^{(k_n-i)}_{X^n_i} \text{ for $i\geq0$}.
\end{equation}
In this way, the $(k_n-i)$-th copy of the~$S_n$ is used to determine what 
happens in round~$i+1$ of the process. Note that the index $k_n-i$ becomes 
negative when $i>k_n$. Our setup is such that if $k_m = k_n$, then the 
shooting processes started from $m$ and~$n$ are coupled from the first 
shooting round onward, but if $k_m = k_n + l$ with $l>0$, then the shooting 
process started from~$m$ will first undergo $l$~independent shooting rounds 
before it becomes coupled with the shooting process started from~$n$. Thus, by 
varying the~$k_n$, we can choose after how many rounds shooting processes with 
different starting points become coupled.

\subsection{Tail bounds}
\label{sec:tails}

In Corollary~\ref{cor:coupling} we have given a bound on the probability that 
the shooting process, starting at any point in some interval, visits another 
interval after one shooting round. This bound is in terms of the tails of the 
distribution of the random variables~$Y_n$. In this section, we show that 
$Y_n$ in fact has the same distribution as a sum of independent Bernoulli 
random variables, and we use this result to obtain tail bounds for~$Y_n$.

\begin{lemma}
	\label{lem:sumofind.}
	For every $n\geq 3$, there exist $n-2$ independent Bernoulli random 
	variables $W_1,\dots,W_{n-2}$ such that $Y_n$ has the same distribution as 
	$W_1 + W_2 + \dots + W_{n-2}$.
\end{lemma}

\begin{proof}
	The proof is based on the following beautiful idea due to Vatutin and 
	Mikha{\u\i}lov~\cite{VatutinMikhailov}: we will show that the generating 
	function of~$Y_n$ has only real roots, and then show that this implies the 
	statement of the lemma. For the first step we observe that 
	$\tbinom{Y_n}{k}$ is just the number of subsets of size~$k$ of the boxes 
	that remain empty after throwing $n-1$~balls into $n-1$~boxes. This 
	implies that
	\[\begin{split}
		\E\binom{Y_n}{k}
		&= \E \sum_{1\leq i_1 < \cdots < i_k \leq n-1} \I(\mbox{boxes 
		$i_1,\dots,i_k$ empty}) \\
		&= \binom{n-1}{k} \P(\mbox{boxes $1,\dots, k$ empty})
		 = \binom{n-1}{k}\left( \frac{n-k-1}{n-1} \right)^{n-1}.
	\end{split}\]
	Hence, if we define
	\[
		R(z) = \sum_{k=0}^{n-1} \binom{n-1}{k} (n-k-1)^{n-1} z^k,
	\]
	then we see that
	\[
		\E\left( z^{Y_n} \right)
		= \E\biggl( \, \sum_{k=0}^{n-1} \binom{Y_n}{k} (z-1)^k \, \biggr)
		=(n-1)^{-(n-1)} R(z-1).
	\]
	
	We want to show that $R$ has only real roots, for which it is enough to 
	show that $z^{n-1}R(1/z)$ has only real roots. To show this, we now write
	\[\begin{split}
		z^{n-1}R(1/z)
		&= \sum_{k=0}^{n-1} \binom{n-1}{k} (n-k-1)^{n-1} z^{n-k-1} \\
		&= \sum_{k=0}^{n-1} \binom{n-1}{k} \left(z \frac{d}{dz}\right)^{\!n-1} 
		z^{n-k-1},
	\end{split}\]
	from which it follows that
	\[
		z^{n-1}R(1/z)
		= \left(z \frac{d}{dz}\right)^{\!n-1} \:
			\sum_{k=0}^{n-1} \binom{n-1}{k} z^{n-k-1}
		= \left(z \frac{d}{dz}\right)^{\!n-1} (z+1)^{n-1}.
	\]
	Now observe that if a polynomial $f(z)$ has only real roots, then so do 
	the polynomials $zf(z)$ and~$f'(z)$ (one way to see the latter is to 
	observe that between any two consecutive zeroes of~$f$, there must be a 
	local maximum or minimum). Therefore, our last expression for $z^{n-1} 
	R(1/z)$ above has only real roots and hence so does~$R$.
	
	It follows that the generating function $\E(z^{Y_n})$ of~$Y_n$ has only 
	real roots. Now note that $\E(z^{Y_n})$ is a polynomial of degree $n-2$ 
	which cannot have positive roots. Let its roots be $-d_1, -d_2, \dots, 
	-d_{n-2}$, with all the $d_i\geq 0$, and let $W_1,\dots,W_{n-2}$ be 
	independent Bernoulli random variables such that
	\[
		\P(W_i=1) = 1-\P(W_i=0) = \frac{1}{1+d_i},
		\qquad i=1,\dots,n-2.
	\]
	Note that these are properly defined random variables because of the fact 
	that $d_i\geq 0$ for all~$i$. Writing $W = W_1+\dots+W_{n-2}$, we now 
	have
	\[
		\E\left( z^W \right)
		= \prod_{i=1}^{n-2} \E\left( z^{W_i} \right)
		= \prod_{i=1}^{n-2} \frac{z+d_i}{1+d_i}
		= \E\left( z^{Y_n} \right),
	\]
	so $Y_n$ and~$W$ have the same distribution.
\end{proof}

Tail bounds for sums of independent Bernoulli random variables are generally 
derived from a fundamental bound due to Chernoff~\cite{Chernoff} by means of 
calculus, see e.g.~\cite{AlonSpencer}*{Appendix~A}. Here we use the following 
result:

\begin{theorem}
	\label{thm:tails}
	Let $W$ be the sum of~$n$ independent Bernoulli random variables, and let 
	$p = \E W/n$. Then for all $u \geq 0$ we have
	\begin{align}
		\P(W\leq \E W - u)
		&\leq \exp\left( -\frac12 \frac{u^2}{np(1-p)-u(1-2p)/3} \right),
		\label{jansonineq1} \\
		\intertext{and}
		\P(W\geq \E W + u)
		&\leq \exp\left( -\frac12 \frac{u^2}{np(1-p)+u(1-2p)/3} \right).
		\label{jansonineq2}
	\end{align}
\end{theorem}

\begin{proof}
	This theorem has been proved by Janson, see \cite{Janson}*{Theorems 1 
	and~2}. For the convenience of the reader we outline the main steps of the 
	proof of inequality~\eqref{jansonineq2} here. 
	Inequality~\eqref{jansonineq1} follows by symmetry.

	By \cite{AlonSpencer}*{Theorem~A.1.9} we have for all $\lambda>0$,
	\begin{equation}
		\label{A19}
		\P(W \geq \E W+u) < e^{-\lambda pn} (pe^{\lambda} + (1-p))^n 
		e^{-\lambda u}.
	\end{equation}
	By the remark following \cite{AlonSpencer}*{Theorem~A.1.9}, for given 
	$p,n,u$, the value of~$\lambda$ that minimizes the right hand side of 
	inequality~\eqref{A19} is
	\begin{equation}
		\label{optimallambda}
		\lambda = \log \biggl[ \left( \frac{1-p}{p} \right) \left( \frac{u 
		+np}{n-(u+np)} \right) \biggr].
	\end{equation}
	In \cite{AlonSpencer} suboptimal values of~$\lambda$ are substituted 
	into~\eqref{A19} to obtain bounds. Substituting the optimal 
	value~\eqref{optimallambda} into~\eqref{A19}, and letting $q=1-p$ and $x = 
	u/n \in [0,q]$, yields \cite{Janson}*{Inequality~(2.1)}:
	\[
		\P(W \geq \E W +u) \leq \exp\biggl( -n (p+x) \log \frac{p+x}{p} - n 
		(q-x) \log \frac{q-x}{q} \biggr).
	\]
	Following~\cite{Janson} we define, for $0\leq x\leq q$,
	\[
		f(x) = (p+x) \log \frac{p+x}{p} + (q-x) \log \frac{q-x}{q} - 
		\frac{x^2}{2(pq+x(q-p)/3)}.
	\]
	Then $f(0)=f'(0) =0$ and
	\[
		f''(x) = \frac{\frac{1}{3} pq (q-p)^2 x^2 + \frac{1}{27} (q-p)^3 x^3 + 
		p^2 q^2 x^2}{(x+p)(q-x)(pq+x(q-p)/3)^3} \geq 0,
	\]
	for $0\leq x\leq q$. Hence $f(x)\geq 0$ for $0\leq x\leq q$, which 
	proves~\eqref{jansonineq2}.
\end{proof}

\begin{corollary}
	\label{cor:tailY}
	Let $n\geq 4$. Then for all $u\geq 0$,
	\begin{align*}
		\P\left( Y_n \leq \frac{n-5/3}{e} -u \right)
		&\leq \exp\left( -\frac{1}{2} \frac{e^2 u^2}{(n-1)(e-1)} \right) \\
		\intertext{and}
		\P\left( Y_n \geq \frac{n-3/2}{e} +u \right)
		&\leq \exp\left( -\frac{1}{2} \frac{e^2 u^2}{(n-1)(e-1) + ue(e-2)/3} 
		\right).
	\end{align*}
\end{corollary}

\begin{proof}
	Recall that $Y_n$ can be interpreted as the number of empty boxes after 
	randomly throwing $n-1$ balls into $n-1$ boxes. Thus we have
	\begin{equation}
		\label{expectationY}
		\E Y_n = (n-1) \left( 1-\frac{1}{n-1} \right)^{n-1}.
	\end{equation}
	We will bound this expectation using the following two inequalities, which 
	hold for all $u\in(0,1)$:
	\begin{align}
		(1-u)^{1/u} &\leq (1-\textstyle\frac12 u) / e,
		\label{inequalityu1}\\
		(1-u)^{1/u} &\geq (1-\textstyle\frac12 u - \frac12 u^2)/e.
		\label{inequalityu2}
	\end{align}
	To prove these inequalities, we define
	\begin{align*}
		h_1(u) &= u + \log (1-u) - u \log (1-\tfrac12 u), \\
		h_2(u) &= u + \log(1-u) - u \log(1-\tfrac12 u -\tfrac12 u^2).
	\end{align*}
	Then $h_1(0)=h_2(0) =h'_1(0) = h'_2(0) =0$ and moreover
	\[
		h_1''(u) = - \frac{u(5-5u+u^2)}{(1-u)^2(2-u)^2},\qquad
		h_2''(u) = \frac{u(7+2u)}{(1-u)(2+u)^2}.
	\]
	Hence $h_1''(u) < 0$ and $h_2''(u)>0$ for $u\in (0,1)$. Therefore, 
	$h_1(u)<0$ and $h_2(u)>0$ for all $u\in (0,1)$, which implies 
	\eqref{inequalityu1} and~\eqref{inequalityu2}.
	
	By \eqref{expectationY} and~\eqref{inequalityu1}, we have that
	\begin{equation}
		\label{EYupper}
		\E Y_n \leq \frac{n-3/2}{e}.
	\end{equation}
	Similarly, using \eqref{expectationY} and~\eqref{inequalityu2}, we obtain 
	that
	\begin{equation}
		\label{EYlower}
		\E Y_n
		\geq \frac{n-3/2}{e} - \frac{1}{2e(n-1)}
		\geq \frac{n-5/3}{e} \qquad \text{for $n\geq4$}.
	\end{equation}
	Since, by Lemma~\ref{lem:sumofind.}, $Y_n$ has the same distribution as a 
	sum of $n-2$ independent Bernoulli random variables, 
	Theorem~\ref{thm:tails} applies to the~$Y_n$. It follows from 
	\eqref{EYupper} and~\eqref{EYlower} that in applying this theorem to~$Y_n$ 
	for $n\geq4$, we can use that
	\[
		(n-2)p \leq \frac{n-1}{e},\qquad
		1-p \leq \frac{e-1}{e},\qquad
		0\leq 1-2p \leq \frac{e-2}{e},
	\]
	where $p = \E Y_n /(n-2)$. This yields the desired result.
\end{proof}

\subsection{Visiting consecutive intervals}
\label{sec:intervals}

Corollaries \ref{cor:coupling} and~\ref{cor:tailY} together give an explicit 
upper bound on the probability that the shooting process, starting anywhere in 
some interval, visits a given other interval after the next shooting round. In 
this section, we extend this result to more than one round. We give an 
explicit construction of a sequence of intervals $I_0, I_1, \dotsc$ and, using 
Corollaries \ref{cor:coupling} and~\ref{cor:tailY}, we estimate the 
probability that the shooting process successively visits each interval in 
this specific sequence. 

To start our construction, suppose that the (real) numbers $I^-_0,I^+_0 \geq 
2$, with $I^+_0 < eI^-_0$, and a parameter $\gamma\in (0,1]$ are given. Set
\begin{equation}
	\label{s_0}
	s_0 := \sum_{i=1}^\infty \sqrt{i}\,e^{-i/2} = 2.312449444\cdots,
\end{equation}
and define the number~$c_0$ in terms of $I^+_0$, $I^-_0$ and~$\gamma$ by
\begin{equation}
	\label{c_0}
	c_0	:= \left( \sqrt{\smash[b]{I^+_0}}-\sqrt{\smash[b]{I^-_0}} \, \right) 
	\frac{\gamma}{s_0\sqrt{e}}.
\end{equation}
For all $k\geq1$, we now define the real numbers $I^-_k$ and~$I^+_k$ by 
\begin{align}
	\label{left}
	I^-_k &:= I^-_0 e^k \biggl( 1 + c_0 \sqrt{\frac{e}{\smash[b]{I^-_0}}} 
	\sum_{i=1}^k \sqrt{i} \, e^{-i/2} \biggr), \\
	\label{right}
	I^+_k &:= I^+_0 e^k \biggl( 1 - c_0 \sqrt{\frac{e}{\smash[b]{I^+_0}}} 
	\sum_{i=1}^k \sqrt{i} \, e^{-i/2} \biggr),
\end{align}
and we set $I_k := [\floor{I^-_k}, \ceil{I^+_k}]$ for all $k\geq0$. These 
specific choices for $I^+_k$ and~$I^-_k$ may look peculiar, but the reader 
will see in our calculations below why they are convenient. At this point, let 
us just note that our intervals are disjoint (since $I^-_{k+1} > I^-_0 e^{k+1} 
> I^+_0 e^k > I^+_k$) and their lengths are (roughly) given by the relatively 
simple expression
\[
	I^+_k - I^-_k
	= (I^+_0 - I^-_0) e^k \biggl( 1 - \frac{\gamma}{s_0} \sum_{i=1}^k \sqrt{i} 
	\, e^{-i/2} \biggr).
\]
For $\gamma=1$, this reduces to
\[
	I^+_k - I^-_k
	= (I^+_0 - I^-_0) e^k \frac{1}{s_0}
		\sum_{j\geq 1} \sqrt{j+k} \, e^{-(j+k)/2}
	\geq (I^+_0 - I^-_0) e^{k/2},
\]
which shows that the lengths of our intervals~$I_k$ grow to infinity with~$k$.

We want to consider shooting processes starting from any $n\in 
\bigcup_{k=1}^\infty I_k$, and we couple these processes as in~\eqref{Xni}, 
where we take $k_n$ equal to the index of the interval containing~$n$. In this 
way, all shooting processes starting from the same interval are coupled from 
the first round onward, while a shooting process starting from a point 
in~$I_{k+l}$ first undergoes $l$ independent shooting rounds (and, with high 
probability, reaches~$I_k$), before it becomes coupled to a shooting process 
starting from a point in~$I_k$. The following lemma gives an estimate of the 
probability that a shooting process starting from any $n\in 
\bigcup_{k=1}^\infty I_k$ visits each of the intervals $I_{k_n-1}, I_{k_n-2}, 
\dots, I_0$, in that order.

\begin{lemma}
	\label{lem:intervals}
	Let the numbers $I^+_0,I^-_0$ and the parameter $\gamma\in (0,1]$ be 
	given, and define the intervals $I_k$ ($k\geq0$) by \eqref{left} 
	and~\eqref{right}, as explained above. For each $n\in \bigcup_{k=1}^\infty 
	I_k$, let $k_n$ be the index of the interval containing~$n$, and define 
	$X^n_i$ ($i\geq0$) by~\eqref{Xni}. Then
	\begin{multline}
		\label{boundintervals}
		\P\bigl( \text{for some $\textstyle n\in \bigcup_{k=1}^\infty I_k$ and 
		$i\leq k_n$, $X_i^n \not\in I_{k_n-i}$} \bigr) \\
		\leq \frac{1}{e^{c_1}-1} + \frac{1}{e^{c_2}-1}
	\end{multline}
	where
	\[
		c_1 = \frac{ec_0^2/2}
			{(e-1)\bigl( 1+c_0 s_0 \, \sqrt{e\smash[b]{\null/I^-_0}} \bigr)}
		\quad\text{and}\quad
		c_2 = \frac{ec_0^2/2}
			{e-1 - c_0(2e-1) \bigm/ 3\sqrt{\smash[b]{I^+_0}} },
	\]
	with $s_0$ and~$c_0$ defined as in \eqref{s_0} and~\eqref{c_0}. Note that 
	the right hand side of~\eqref{boundintervals} depends, via $c_0,c_1,c_2$, 
	on the choice of $I^+_0$, $I^-_0$ and~$\gamma$.
\end{lemma}

\begin{proof}
	Let the $S_n^{(k)}$, for $n\geq 2$ and $k\geq 1$, be coupled as in 
	Section~\ref{sec:coupling}. Suppose we are on the event that for all 
	$k\geq 1$ and $n\in I_k$ it holds that $S_n^{(k)} \in I_{k-1}$. Then, by 
	our coupling, it follows that for all $k\geq 1$, $n\in I_k$ and $i\leq k$, 
	$X_i^n \in I_{k-i}$. The latter statement is equivalent to saying that for 
	all $n\in \bigcup_{k=1}^{\infty} I_k$ and $i\leq k_n$, $X_i^n \in 
	I_{k_n-i}$. Therefore, the left hand side of~\eqref{boundintervals} is 
	bounded above by
	\begin{equation}
		\label{proofintervals000}
		\P(\exists k\geq 1, \exists n\in I_k\colon S_n^{(k)} \not\in I_{k-1})
		\leq \sum_{k=1}^{\infty} \P(\exists n\in I_k\colon S_n \not\in 
		I_{k-1}).
	\end{equation}

	By Corollary~\ref{cor:coupling} we have for $k\geq 1$,
	\begin{multline}
		\label{proofintervals1}
		\P(\exists n\in I_k\colon S_n\notin I_{k-1}) \\
		\leq \P\bigl( Y_{\floor{I_k^-}} \leq \floor{I_{k-1}^-} - 1 \bigr) +
		\P\bigl( Y_{\ceil{I_k^+}} \geq \ceil{I_{k-1}^+} \bigr).
	\end{multline}
	To bound the right hand side of~\eqref{proofintervals1}, we use 
	Corollary~\ref{cor:tailY}, which applies for all $k\geq 1$ since $I^-_1 
	\geq eI^-_0 > 4$. We first note that since $\tfrac{1+5/3}{e}-1<0$,
	\begin{equation}
		\label{proofintervals1a}
		\floor{I^-_{k-1}} - 1 - \frac{\floor{I^-_k} - 5/3}{e}
		\leq I^-_{k-1} - \frac{I^-_k}{e}
		= -c_0 \sqrt{\smash[b]{I^-_0}} \, e^{(k-1)/2} \sqrt{k}.
	\end{equation}
	By Corollary~\ref{cor:tailY} with $n = \floor{I_k^-}$ and $-u$ equal to 
	the right hand side of~\eqref{proofintervals1a}, using $\floor{I^-_k} - 1 
	\leq I^-_0 e^k \bigl( 1+c_0 s_0\sqrt{e\smash[b]{\null/I^-_0}} \, \bigr)$, 
	we see that
	\begin{equation}
		\label{proofintervals2}
		\P( Y_{\floor{I_k^-}} \leq \floor{I_{k-1}^-} - 1 )
		\leq \exp\left( -\frac{1}{2} \frac{ec_0^2 k}
			{(e-1)\bigl( 1+c_0 s_0 \, \sqrt{e\smash[b]{\null/I^-_0}} \bigr)} 
			\right).
	\end{equation}
	Likewise,
	\[
		\ceil{I^+_{k-1}} - \frac{\ceil{I^+_k} - 3/2}{e}
		\geq I^+_{k-1} - \frac{I^+_k}{e}
		= c_0 \sqrt{\smash[b]{I^+_0}} \, e^{(k-1)/2} \sqrt{k}.
	\]
	By Corollary~\ref{cor:tailY}, using $\ceil{I^+_k}-1 \leq I^+_0 e^k - c_0 
	e^k \sqrt{\smash[b]{I^+_0}}$ and $e^{(k-1)/2} \sqrt{k} \leq e^{k-1}$, 
	\begin{equation}
		\label{proofintervals3}
		\P( Y_{\ceil{I_k^+}} \geq \ceil{I_{k-1}^+} )
		\leq \exp\left( -\frac{1}{2} \frac{ec_0^2 k}
			{e-1 - c_0 (2e-1) \bigm/ 3\sqrt{\smash[b]{I^+_0}} } \right).
	\end{equation}
	By \eqref{proofintervals1}, \eqref{proofintervals2} 
	and~\eqref{proofintervals3}, the right hand side 
	of~\eqref{proofintervals000} is bounded above by the sums over all 
	$k\geq1$ of the right hand sides of \eqref{proofintervals2} 
	and~\eqref{proofintervals3}, added together. This 
	proves~\eqref{boundintervals}.
\end{proof}

\section{Non-convergence}
\label{sec:nonconvergence}

In this section we prove non-convergence of the~$p_n$:

\begin{theorem}[Non-convergence]
	\label{thm:nonconvergence}
	It is the case that
	\[
		\limsup_{n \to \infty} p_n \geq \supf
		\qquad\text{and}\qquad
		\liminf_{n \to \infty} p_n \leq \inff.
	\]
\end{theorem}

The idea of the proof of Theorem \ref{thm:nonconvergence} is as follows. We 
will take intervals $H_0$ and~$V_0$ around $n=2795$ and~$n=4608$, the last 
peak and valley in Figure~\ref{figure1}, respectively, so that $p_n$ is high 
on~$H_0$ and low on~$V_0$. Then we will construct sequences of intervals $H_1, 
H_2, \dotsc$ and $V_1, V_2, \dotsc$ in such a way that, if $n\in H_k$ for 
some~$k$, then with high probability uniformly in~$k$, the number of survivors 
in the shooting process starting with $n$~persons will, during the first~$k$ 
shooting rounds, visit each of the intervals $H_{k-1}, H_{k-2}, \dots, H_0$ 
(in that order), and similar for~$V_k$. As a consequence, $p_n$ must be high 
on all intervals~$H_k$, and low on all intervals~$V_k$.

To make this work, the intervals $H_0$ and~$V_0$ should be big enough to make 
the probability high that the number of survivors after $k$~rounds will lie in 
them when we start from $H_k$ or~$V_k$, but small enough so that the values 
taken by the~$p_n$ on the respective intervals $H_0$ and~$V_0$ are 
sufficiently separated from each other. It turns out that $H_0 = [2479, 3151]$ 
and $V_0 = [4129, 5143]$ work, and these intervals form our starting point.

The next three intervals $H_1$, $H_2$ and~$H_3$ are constructed as follows. We 
choose the right boundary~$H^+_1$ of~$H_1$ such that $\E S_{H_1^{\smash+}}$ 
lies roughly $3.56$~standard deviations away from the right boundary of~$H_0$, 
and we choose the left boundary~$H_1^-$ of~$H_1$ similarly. In this way, we 
expect that after one shooting round we will end up in~$H_0$ with high 
probability, when we start in~$H_1$. The intervals $H_2$ and~$H_3$ are 
constructed similarly, and so are the intervals $V_1$ and~$V_2$. We need this 
special treatment only for two (instead of three) intervals $V_1$ and~$V_2$, 
because $V_0$ lies to the right of~$H_0$. We end up with
\begin{align*}
	H_0 &= [2479, 3151],	& V_0 &= [4129, 5143],\\
	H_1 &= [6991, 8290],	& V_1 &= [11553, 13623], \\
	H_2 &= [19425, 22086],	& V_2 &= [31952, 36447], \\
	H_3 &= [53501, 59301].
\end{align*}

The remaining intervals are now constructed as explained in 
Section~\ref{sec:intervals}, taking $H_3$ and~$V_2$ as the respective starting 
intervals. To be more precise, we first set $I_0 := H_3$, take $\gamma=1$, and 
then for $k\geq4$ define the intervals $H_k = [H^-_k, H^+_k] := 
[\floor{I^-_{k-3}}, \ceil{I^+_{k-3}}]$ using equations \eqref{left} 
and~\eqref{right} for the endpoints. In the same way we define the intervals 
$V_k$ for $k\geq3$, taking $I_0 := V_2$ as the initial interval in the 
construction from Section~\ref{sec:intervals}.

The following lemma tells us that the values of the~$p_n$ on the intervals 
$H_0$ and~$V_0$ are sufficiently separated from each other and that, when we 
start in~$H_3$, the number of survivors in the shooting process will visit 
each of the intervals $H_2, H_1, H_0$ with high probability, and similarly for 
$V_2, V_1, V_0$. We obtain the desired bounds using computations in 
\Mathematica. We explain how we can perform the computations in such a way 
that we avoid introducing rounding errors, and thus obtain rigorous results.

\begin{lemma}
	\label{lem:mathematica}
	We have
	\begin{align*}
		\min\{ p_n\colon n \in H_0\} &\geq 0.5163652651, \\
		\max\{ p_n\colon n \in V_0\} &\leq 0.4767018688,
	\end{align*}
	and moreover
	\begin{align*}
		\sum_{k=1}^3 \left[ \P(Y_{H_k^-} \leq H_{k-1}^- - 1) + \P(Y_{H_k^+} 
		\geq H_{k-1}^+) \right] &\leq 0.0010954222, \\
		\sum_{k=1}^2 \left[ \P(Y_{V_k^-} \leq V_{k-1}^- - 1) + \P(Y_{V_k^+} 
		\geq V_{k-1}^+) \right] &\leq 0.0006060062.
	\end{align*}
\end{lemma}

\begin{proof}
	The explicit bounds in the first part of Lemma~\ref{lem:mathematica} are 
	based on exact calculations in \Mathematica\ of bounds on the 
	numbers~$p_n$ up to $n=6000$ using the recursion
	\begin{equation}
		\label{recursion}
		p_n = \sum_{k=0}^{n-2} \P(S_n=k)\,p_k \qquad (n\geq2),
	\end{equation}
	with $p_0=1$ and $p_1=0$. To obtain lower bounds on~$p_n$ from this 
	recursion, we need lower bounds on the $\P(S_n=k)$. To this end, write
	\[
		t^n_{k,r} = \binom{n}{k} \binom{n-k}{r} (n-k-r)^{k+r} 
		(n-k-r-1)^{n-k-r}
	\]
	for the terms that appear in the inclusion-exclusion 
	formula~\eqref{Sn-incl-excl}. Observe that these are integer numbers. Now, 
	for fixed $n$ and~$k$, define $r_{\max}$ by
	\[
		r_{\max} := \min\{ r\geq0 \colon 10^{10} t^n_{k,2r} < (n-1)^n \}.
	\]
	Since truncating the sum in the inclusion-exclusion formula after an even 
	number of terms yields a lower bound on~$\P(S_n=k)$, we have that
	\[
		\P(S_n=k) \geq (n-1)^{-n} \sum_{r=0}^{2r_{\max}-1} (-1)^r t^n_{k,r}.
	\]
	By our choice of~$r_{\max}$, we know that the difference between the left 
	and right hand sides of this inequality is smaller than~$10^{-10}$.

	However, this rational lower bound on~$\P(S_n=k)$ is numerically awkward 
	to work with, because the numerator and denominator become huge for 
	large~$n$. We therefore bound $\P(S_n=k)$ further by the largest smaller 
	rational number of the form $m/10^{10}$ with $m\in\N$. Stated in a 
	different way, we bound the quantity $10^{10} \P(S_n=k)$ from below by the 
	integer
	\[
		P_{n,k} := 0 \vee \floor[\bigg]{ 10^{10} \sum_{r=0}^{2r_{\max}-1} 
		(-1)^r t^n_{k,r} \biggm/ (n-1)^n },
	\]
	where we remark that for integers $a$ and~$b$, $\floor{a/b}$ is just the 
	quotient of the integer division~$a/b$.

	We now return to~\eqref{recursion}. Suppose that we are given nonnegative 
	integers $\hat{p}_0, \hat{p}_1, \dots, \hat{p}_{n-1}$ that satisfy 
	$10^{10}p_k \geq \hat{p}_k$ for $k = 0,1,\dots,n-1$. Let
	\[
		\hat{p}_n := \floor[\bigg]{\, \sum_{k = k_1}^{k_2} P_{n,k} \hat{p}_k 
		\biggm/ 10^{10} },
	\]
	where
	\[
		k_1 = 0 \vee \ceil[\big]{n/e - \sqrt{5n}\,}, \quad
		k_2 = (n-2) \wedge \floor[\big]{n/e + \sqrt{5n}\,}.
	\]
	Then it follows from~\eqref{recursion} and the fact that $10^{10} 
	\P(S_n=k) \geq P_{n,k}$, that $10^{10}p_n \geq \hat{p}_n$. In this way, 
	starting from the values $\hat{p}_0 = 10^{10}$ and $\hat{p}_1 = 0$, we 
	recursively compute integer lower bounds on the numbers $10^{10} p_n$, or 
	equivalently, rational lower bounds on~$p_n$, up to $n=6000$. We emphasize 
	that this procedure involves only integer calculations, that could in 
	principle be done by hand. For practical reasons, we invoke the aid of 
	\Mathematica\ to perform these calculations for us, using exact integer 
	arithmetic.

	In the same way (now starting the recursion from $\hat{p}_0 = 0$ and 
	$\hat{p}_1 = 10^{10}$), we compute exact bounds on the probabilities 
	$1-p_n$ of ending up with a single survivor. Taking complements, this 
	gives us rational upper bounds on the~$p_n$ up to $n = 6000$. The first 
	part of Lemma~\ref{lem:mathematica} follows from these exact bounds, and 
	Figure~\ref{figure1} shows the lower bounds as a function of~$\log n$. As 
	it turns out, the largest difference between our upper and lower bounds on 
	the~$p_n$ is~$527 \times 10^{-10}$.

	The second part of Lemma~\ref{lem:mathematica} again follows from exact 
	integer calculations with the aid of \Mathematica. Inclusion-exclusion 
	tells us that
	\begin{equation}
		\label{Yincl-excl}
		\P(Y_{n+1} = k+i) = \frac1{n^n} \sum_{r=i}^{n-k} (-1)^{r-i} 
		\binom{n}{k+i} \binom{n-k-i}{r-i} (n-k-r)^n.
	\end{equation}
	Summing over $i = 0,\dots,n-k$, interchanging the order of summation, and 
	reorganising the binomial coefficients yields
	\[
		\P(Y_{n+1} \geq k) = \frac1{n^n} \sum_{r=0}^{n-k} (-1)^r \sum_{i=0}^r 
		(-1)^i \binom{k+r}{k+i} \binom{n}{k+r} (n-k-r)^n.
	\]
	Using the binomial identity
	\[
		\sum_{i=0}^r (-1)^i \binom{k+r}{k+i} = \binom{k+r-1}{r}
		\qquad (k\geq1, r\geq0),
	\]
	which is easily proved by induction in~$r$, we conclude that for $k \geq 
	1$,
	\[
		\P(Y_{n+1} \geq k) = \frac1{n^n} \sum_{r=0}^{n-k} (-1)^r \frac{k}{k+r} 
		\binom{n}{k} \binom{n-k}{r} (n-k-r)^n.
	\]
	Since $\P(Y_{n+1} \leq k) = 1 - \P(Y_{n+1} \geq k) + \P(Y_{n+1} = k)$, the 
	previous equation together with~\eqref{Yincl-excl} for $i=0$ gives
	\[
		\P(Y_{n+1} \leq k) = 1 + \frac1{n^n} \sum_{r=0}^{n-k} (-1)^r 
		\frac{r}{k+r} \binom{n}{k} \binom{n-k}{r} (n-k-r)^n.
	\]

	We note that from our derivation it follows that, as before, the terms 
	that appear in the sums above are integers. This allows us to compute the 
	rational numbers $\P(Y_{n+1} \leq k)$ and $\P(Y_{n+1} \geq k)$, and hence 
	the sums in the second part of Lemma~\ref{lem:mathematica}, using only 
	exact integer arithmetic. Bounding these sums above by rational numbers of 
	the form $m/10^{10}$ (which again involves only integer arithmetic) yields 
	the second part of Lemma~\ref{lem:mathematica}.
\end{proof}

\begin{proof}[Proof of Theorem~\ref{thm:nonconvergence}.]
	Let the intervals~$H_k$, for $k\geq 0$, be constructed as explained below 
	the statement of Theorem~\ref{thm:nonconvergence}. Similarly as in 
	Section~\ref{sec:intervals}, for $n\in \bigcup_{k=1}^\infty H_k$ we now 
	define $X^n_i$ by~\eqref{Xni}, with $k_n$ equal to the value of~$k$ such 
	that $n\in H_k$. Recall that $X^n_i$ represents the number of survivors 
	after round~$i$ of the shooting process started from~$n$. Fix a $k\geq1$ 
	and $n\in H_k$. We are interested in the event
	\[
		G_n	= \{ \text{$X_i^n\in H_{k-i}$ for all $i=1,\dots,k$} \}.
	\]
	It follows from
	\[\begin{split}
		\P(G_n^c)
		& = \P(\exists i\leq k\colon X_i^n \not\in H_{k-i} ) \\
		& \leq \P(\exists i\leq k-3\colon X_i^n \not\in H_{k-i} )
			+ \sum_{k=1}^3 \P( \exists m\in H_k\colon S_m\not\in H_{k-1})
	\end{split}\]
	and Corollary~\ref{cor:coupling}, that $\P(G_n^c)$ is bounded from above 
	by
	\begin{multline}
		\label{nonconv10}
		\P(\exists i\leq k-3\colon X_i^n \not\in H_{k-i} ) \\
		+ \sum_{k=1}^3 \left[ \P\bigl( Y_{H_k^-} \leq H_{k-1}^- - 1 \bigr) + 
		\P\bigl( Y_{H_k^+} \geq H_{k-1}^+ \bigr) \right].
	\end{multline}

	We use Lemma~\ref{lem:intervals} to compute an upper bound on the first 
	term in~\eqref{nonconv10}, and Lemma~\ref{lem:mathematica} to bound the 
	sum in the second term. This gives
	\[
		\P(G_n^c) \leq  0.0007188677 + 0.0010954222 = 0.0018142899,
	\]
	uniformly for all $k\geq1$ and $n\in H_k$. Using the first part of 
	Lemma~\ref{lem:mathematica}, this gives
	\[
		p_n \geq \P(G_n) \min_{m\in H_0} p_m
		\geq 0.9981857101 \times 0.5163652651 \geq \supf,
	\]
	for all $n\in \bigcup_{k=1}^\infty H_k$. In a similar way, we bound the 
	values $1-p_n$ from below, and hence the~$p_n$ from above, on the 
	intervals~$V_k$.
\end{proof}

\section{Periodicity and continuity}
\label{sec:periodicity}

\subsection{Main theorem}
\label{sec:maintheorem}

In this section we prove the convergence of the~$p_n$ on the $\log n$~scale to 
a periodic and continuous function~$f$. Together with 
Theorem~\ref{thm:nonconvergence} (non-convergence), this gives 
Theorem~\ref{thm:main}.

\begin{theorem}[Asymptotic periodicity and continuity]
	\label{thm:asymptotics}
	There exists a periodic and continuous function $f\colon \R\to[0,1]$ of 
	period~1 such that
	\[
		\sup_{x\geq x_0} \, \abs[\big]{ p_{\floor{\exp x}} - f(x) } \to 0
		\qquad \text{as }x_0\to\infty.
	\]
\end{theorem}

To prove Theorem~\ref{thm:asymptotics}, we consider coupled shooting processes 
started from different points that lie in one of the intervals
\begin{align}
	\label{J0}
	J_0 &= \bigl[ e^{k_0+w-3\delta},
				e^{k_0+w-3\delta} + (e^{k_0+w-3\delta})^{2/3} \bigr], \\
	\label{Jk}
	J_k &= \bigl[ e^{k_0+w+k-\delta}, e^{k_0+w+k+\delta} \bigr],
			\qquad k\geq 1,
\end{align}
for some $k_0$, $w$ and~$\delta$ specified in Proposition~\ref{prop:ingr3} 
below. Observe that the intervals~$J_k$ for $k\geq1$ have length~$2\delta$ on 
the $\log n$ scale. We will show in three steps that with high probability, 
the distance between the numbers of survivors in these shooting processes 
decreases, and the coupled processes collide before the number of survivors 
has reached 0 or~1. The three steps are respectively described by Propositions 
\ref{prop:ingr3}, \ref{prop:ingr2} and \ref{prop:ingr1} below.

\begin{proposition}
	\label{prop:ingr3}
	For all $\varepsilon>0$ and~$a_2$ there exist $\delta\in(0,\tfrac13)$ and 
	$k_0 > 1+\log a_2$ such that, with the intervals~$J_k$ as in \eqref{J0} 
	and~\eqref{Jk},
	\[
		\inf_{w\in[0,1]}  \P\bigl( \text{for all $\textstyle n \in 
		\bigcup_{k=1}^{\infty} J_k$, $X^n_{k_n} \in J_0$} \bigr)
		\geq 1-\varepsilon,
	\]
	where for all $n\in \bigcup_{k=1}^\infty J_k$ and $i\geq0$, $X^n_i$ is 
	defined by~\eqref{Xni}, with $k_n$ equal to the value of~$k$ such that 
	$n\in J_k$.
\end{proposition}

Note that on the event considered in Proposition~\ref{prop:ingr3}, the number 
of survivors~$X^n_{k_n}$ after $k_n$ shooting rounds for different starting 
points $n\in \bigcup_{k=1}^\infty J_k$ are all in the same interval~$J_0$. 
By~\eqref{Xni}, from this moment onward the processes~$X^n_{k_n+i}$ ($i\geq0$) 
for different~$n$ will be coupled together. Our next two propositions explore 
what will happen when we are in a situation like this.

\begin{proposition}
	\label{prop:ingr2}
	For all $n\geq2$ and $i\geq0$, let the~$X^n_i$ be coupled as 
	in~\eqref{Xni}, with $k_n=0$ for all~$n$. Then for all $\varepsilon>0$ 
	there exist $a_0$ and~$d$ such that, for all $a,b$ with $a_0\leq a<b\leq 
	a+a^{2/3}$,
	\[
		\P\bigl( \text{$-1 \leq  X^b_i - X^a_i  \leq d$ and $X_i^a \geq 
		a^{0.01}$ for some~$i$} \bigr) \geq 1-\varepsilon.
	\]
\end{proposition}

\begin{proposition}
	\label{prop:ingr1}
	For all $n\geq2$ and $i\geq0$, let the~$X^n_i$ be coupled as 
	in~\eqref{Xni}, with $k_n=0$ for all~$n$. Then for all $\varepsilon>0$ 
	and~$d$ there exists~$a_1$ such that, for all $a,b$ with $a_1 \leq a < b 
	\leq a+d$,
	\[
		\P\bigl( \text{$X^a_i = X^b_i$ for some~$i$} \bigr)
		\geq 1-\varepsilon.
	\]
\end{proposition}

We will now prove Theorem~\ref{thm:asymptotics} using these three 
propositions, and defer the proofs of Propositions \ref{prop:ingr3}, 
\ref{prop:ingr2} and~\ref{prop:ingr1} to Sections \ref{sec:oneshootinground} 
and~\ref{sec:proofpropositions}.

\begin{proof}[Proof of Theorem~\ref{thm:asymptotics}.]
	We define, for all $x\geq 0$ and integer~$k$,
	\[
		f_k(x) := p_{\floor{\exp(k+x)}}.
	\]
	First we will use Propositions \ref{prop:ingr3}, \ref{prop:ingr2} 
	and~\ref{prop:ingr1} to prove that for all $\varepsilon>0$, there exist 
	$\delta>0$ and~$k_0$ such that, for all $u,v\in [0,1]$ with $\abs{u-v}\leq 
	\delta$,
	\begin{equation}
		\label{asymptoticsstar}
		\abs{f_k(u) - f_l(v)} \leq \varepsilon
		\quad\text{for all $k,l\geq k_0$}.
	\end{equation}
	
	Let $\varepsilon>0$. Choose $a_0$ and~$d$ according to 
	Proposition~\ref{prop:ingr2} such that, for all~$a,b$ with $a_0\leq 
	a<b\leq a+a^{2/3}$,
	\begin{equation}
		\label{mainingr2}
		\P\bigl( \text{$\abs{X_i^b - X_i^a} \leq d$ and $X_i^a, X_i^b \geq 
		a^{0.01} - 1$ for some~$i$} \bigr) \geq 1-\frac{\varepsilon}{3}.
	\end{equation}
	Next, choose~$a_1$ according to Proposition~\ref{prop:ingr1} such that, 
	for all~$a,b$ with $a_1\leq a < b\leq a+d$,
	\begin{equation}
		\label{mainingr1}
		\P\bigl( \text{$X_i^a = X_i^b$ for some~$i$} \bigr)
		\geq 1 - \frac{\varepsilon}{3}.
	\end{equation}
	Recall that in both \eqref{mainingr2} and~\eqref{mainingr1}, the shooting 
	processes~$X^n_i$ are coupled from the first shooting round onward.

	Finally, we define $a_2 := \max\{ a_0, (a_1+1)^{100} \}$, and choose 
	$\delta \in (0,\tfrac13)$ and $k_0 > 1 + \log a_2$ according to 
	Proposition~\ref{prop:ingr3} such that
	\begin{equation}
		\label{mainingr3}
			\inf_{w\in[0,1]} \P\bigl( \text{for all $\textstyle n \in 
			\bigcup_{k=k_0}^\infty J_k$, $X^n_{k_n} \in J_0$} \bigr)
			\geq 1-\frac{\varepsilon}{3},
	\end{equation}
	where the $X^n_i$ are coupled as in~\eqref{Xni}, with $k_n$ equal to the 
	index of the interval~$J_k$ containing~$n$. We claim 
	that~\eqref{asymptoticsstar} holds for these $\delta$ and~$k_0$.

	In order to prove this, let $u,v \in[0,1]$ be such that $\abs{u-v}\leq 
	\delta$ and let $k,l\geq k_0$. Write $w = (u+v)/2$ and set
	\[
		\alpha := \floor{\exp(k+u)},\qquad \beta := \floor{\exp(l+v)}.
	\]
	Note from \eqref{J0} and~\eqref{Jk} that $\alpha\in J_{k-k_0}$ and 
	$\beta\in J_{l-k_0}$, so in particular, $X^\alpha_i$ and~$X^\beta_i$ are 
	defined and coupled as described above. We need to show that $\abs{ 
	p_\alpha - p_\beta } \leq \varepsilon$, but we will actually prove the 
	stronger statement that
	\begin{equation}
		\label{xsequal}
		\P \bigl( \text{$X^\alpha_{k_\alpha+i} = X^\beta_{k_\beta+i}$ for 
		some~$i$} \bigr) \geq 1-\varepsilon.
	\end{equation}
	
	To prove~\eqref{xsequal}, first note that by \eqref{mainingr3} 
	and~\eqref{J0},
	\begin{equation}
		\label{mainingr3c}
			\P\bigl( X^a_{k_a}, X^b_{k_b} \in \bigl[ e^{k_0+w-3\delta}, 
			e^{k_0+w-3\delta} + (e^{k_0+w-3\delta})^{2/3} \bigr] \bigr)
			\geq 1-\frac{\varepsilon}{3}.
	\end{equation}
	Since $k_0 > 1+\log a_2$ and $\delta < 1/3$, we have that 
	$e^{k_0+w-3\delta} \geq a_2\geq a_0$. Using the fact that 
	$X^\alpha_{k_\alpha+i}$ and~$X^\beta_{k_\beta+i}$ are coupled together for 
	all~$i\geq0$, and since $a_2^{0.01} \geq a_1+1$, it now follows from 
	\eqref{mainingr3c} and~\eqref{mainingr2} that
	\begin{multline}
		\label{mainingr2c}
		\P\bigl( \text{$\abs[\big]{X^\alpha_{k_\alpha+i} - 
		X^\beta_{k_\beta+i}} \leq d$ and $X^\alpha_{k_\alpha+i}, 
		X^\beta_{k_\beta+i} \geq a_1$ for some~$i$} \bigr) \\
		\geq 1 - \frac{2\varepsilon}{3}.
	\end{multline}
	By \eqref{mainingr2c} and~\eqref{mainingr1}, we have that~\eqref{xsequal} 
	holds. This proves~\eqref{asymptoticsstar}.

	Next we prove that \eqref{asymptoticsstar} implies the theorem. Let 
	$\varepsilon>0$, and let $\delta>0$ and~$k_0$ be such that 
	\eqref{asymptoticsstar} holds for this~$\varepsilon$. Fix $x\geq 0$. 
	Taking $u = v = x-\floor{x}$ in~\eqref{asymptoticsstar} and using $f_k(x) 
	= f_k(\floor{x} + u) = f_{k+\floor{x}}(u)$, we get
	\begin{equation}
		\label{asymptoticsone}
		\abs{f_k(x) - f_l(x)}
		= \abs[\big]{f_{k+\floor{x}}(u) - f_{l+\floor{x}}(u)}
		\leq \varepsilon,\quad \text{for all $k,l\geq k_0$}.
	\end{equation}
	In particular, $\sup_{k\geq k_0} f_k(x) \leq \varepsilon + \inf_{k\geq 
	k_0} f_k(x)$ and hence $\lim_{k\to\infty} f_k(x)$ exists. We define
	\[
		f(x) := \lim_{k\to\infty} f_k(x), \qquad x\geq0.
	\]
	Since $f_k(l+x) = f_{k+l}(x)$ for integer~$k,l$, the limit function~$f$ is 
	periodic with period~1. Furthermore, since \eqref{asymptoticsone} holds 
	uniformly for all $x\geq 0$, by taking $k=k_0$ and letting $l\to\infty$ we 
	obtain
	\[
		\varepsilon
		\geq \sup_{x\geq 0} \abs{f_{k_0}(x) - f(x)}
		= \sup_{x\geq 0} \abs[\big]{p_{\floor{\exp(k_0+x)}} - f(k_0+x)},
	\]
	which proves the desired uniform convergence to the limit function~$f$. 
	Finally, by~\eqref{asymptoticsstar} we obtain that, for all $u,v\in [0,1]$ 
	with $\abs{u-v}\leq \delta$,
	\[
		\abs{f(u) - f(v)}
		= \lim_{k\to\infty} \abs{f_k(u) - f_k(v)}
		\leq \varepsilon,
	\]
	which shows that $f$ is continuous. This completes the proof.
\end{proof}

\subsection{One shooting round}
\label{sec:oneshootinground}

In this section we give two key ingredients for the proof of Propositions 
\ref{prop:ingr3}, \ref{prop:ingr2} and~\ref{prop:ingr1}. These two ingredients 
give information about one shooting round of coupled shooting processes 
starting at two different points $a$ and~$b$. The first ingredient is the 
following lemma:

\begin{lemma}
	\label{lem:ingr1}
	Let $S_n$, $n\geq 2$, be coupled as in Section~\ref{sec:coupling}.	For 
	all $a,b$ such that $a< b \leq \tfrac{5}{4} a$,
	\[
		\P(S_a=S_b) \geq e^{-7(b-a)}.
	\]
\end{lemma}

\begin{proof}
	Let $Y_i^n$, $S_i^n$ and~$Z_i^n$ be coupled as in 
	Section~\ref{sec:coupling}. By Lemma~\ref{lem:coupling}, we have that 
	$Y_a^a \leq S_a \leq Z^b_b$ and $Y_a^a \leq S_b \leq Z^b_b$. Hence it 
	suffices to show that
	\begin{equation}
		\label{ingr1.1}
		\P(Y^a_a = Z^b_b) \geq e^{-7(b-a)}.
	\end{equation}
	Note that the Markov chain~$(Z^b_i)_i$ first takes $b-a$ steps 
	independently, before its steps are coupled to the Markov 
	chain~$(Y^a_i)_i$. To prove~\eqref{ingr1.1}, we will first estimate the 
	probability that the~$Z^b$ process decreases to the height~$a$ in these 
	first $b-a$ steps, and then estimate the probability that in the remaining 
	$a$~steps, the distance between $Z^b_{b-a+i}$ and~$Y^a_i$ never increases.

	For the first part, note that by Robbins' version of Stirling's 
	formula~\cite{Robbins},
	\[\begin{split}
		\P(Z^b_{b-a} = a)
		&= \frac{b-1}{b-1} \frac{b-2}{b-1} \dotsm \frac{a}{b-1}
		 = \frac{(b-1)!}{(a-1)!} (b-1)^{-(b-a)} \\
		 &\geq \frac{\sqrt{2\pi} (b-1)^{b-1+\frac{1}{2}} e^{-(b-1)} 
		 e^{1/(12b-11)} } {\sqrt{2\pi} (a-1)^{a-1+\frac{1}{2}} e^{-(a-1)} 
		 e^{1/(12a-12)}} (b-1)^{-(b-a)},
	\end{split}\]
	which implies
	\begin{equation}
		\label{ingr1.2}
		\P(Z^b_{b-a} = a) \geq e^{-(b-a)}.
	\end{equation}

	Next we consider the probability that in the remaining steps, the 
	processes $Y^a_i$ and~$Z^b_{b-a+i}$ stay at the same height. By the 
	coupled transition probabilities \eqref{Yprocess} and~\eqref{Zprocess}, 
	for all $i = 0,1,\dots,a-1$ and $k < a$ we have
	\begin{multline*}
		\P(Y_{i+1}^a = Z_{b-a+i+1}^b \mid Y^a_i = Z^b_{b-a+i} = k)
		= 1- \frac{k}{a-1} + \frac{k-1}{b-1} \\
		\geq 1 - \frac{a}{a-1} +\frac{a-1}{b-1}
		= 1-\frac{1}{a-1} - \frac{b-a}{b-1}
		\geq 1 - \frac{2(b-a)}{a} - \frac{b-a}{a}.
	\end{multline*}
	By our assumption that $b \leq \tfrac54 a$ and the inequality $1-u \geq 
	e^{-2u}$, which holds for $0\leq u\leq \tfrac34$, this gives
	\[
		\P(Y_{i+1}^a = Z_{b-a+i+1}^b \mid Y^a_i = Z^b_{b-a+i} = k)
		\geq e^{-6(b-a)/a}.
	\]
	A separate computation shows that this bound also holds for $k=a$. Since 
	this bound holds for each of the remaining $a$~steps, we conclude that
	\begin{equation}
		\label{ingr1.3}
		\P(Y_a^a = Z_b^b \mid Z^b_{b-a} = a) \geq e^{-6(b-a)}.
	\end{equation}
	Together with~\eqref{ingr1.2}, this gives~\eqref{ingr1.1}, which completes 
	the proof.
\end{proof}

The second key ingredient is Lemma~\ref{lem:ingr2} below. For the proof, we 
need the following preliminary result:

\begin{lemma}
	\label{lem:binomial}
	Let $\lambda_1,\lambda_2 > 0$ be such that $\lambda_1/\lambda_2$ is an 
	integer. Let $T_x$, $x=0,1,\dots,\lambda_1/\lambda_2$, be independent 
	random variables such that $T_x$ has the exponential distribution with 
	parameter $\lambda_1 - x \lambda_2$ (where $T_{\lambda_1/\lambda_2} = 
	\infty$ with probability~1). Define
	\[
		X_t = \min\{ x\colon T_0 + T_1 + \dots + T_x > t \}.
	\]
	Then $X_t$ has the binomial distribution with parameters $n= 
	\lambda_1/\lambda_2$ and $p = 1-e^{-\lambda_2 t}$.
\end{lemma}

\begin{proof}
	Let $n = \lambda_1/\lambda_2$. Consider $n$ independent Poisson processes, 
	each with rate~$\lambda_2$. Let $X_t'$ be the number of Poisson processes 
	that have at least~1 jump before time~$t$. Clearly, $X_t'$ has the 
	binomial distribution with parameters $n = \lambda_1/\lambda_2$ and $p = 
	1-e^{-\lambda_2 t}$. We will prove that $X_t'$ has the same law as~$X_t$, 
	which implies the statement of the lemma.

	Let $T_0'$ be the waiting time until one of the $n$ Poisson processes has 
	a jump. Then $T_0'$ has the exponential distribution with parameter $n 
	\lambda_2 = \lambda_1$, hence $T_0'$ has the same law as~$T_0$. Without 
	loss of generality, suppose this jump occurs in Poisson process~1. Let 
	$T_1'$ be the waiting time from time~$T_0'$ until one of the Poisson 
	processes 2 through~$n$ has a jump. Then $T_1'$ has the exponential 
	distribution with parameter $(n-1) \lambda_2 = \lambda_1 - \lambda_2$, 
	hence $T_1'$ has the same law as~$T_1$. Moreover, $T_0'$ and~$T_1'$ are 
	independent. Continuing in this way, we construct independent random 
	variables $T_0', T_1', \dots, T_{n-1}'$ that have the same laws as $T_0, 
	T_1, \dots, T_{n-1}$. Finally, we define $T_n' = \infty$. We then have 
	that $X_t' = \min\{ x\colon T_0' + T_1' + \dots + T_x' > t \}$. It follows 
	that $X_t'$ has the same law as~$X_t$.
\end{proof}

\begin{lemma}
	\label{lem:ingr2}
	Let $S_n$, $n\geq 2$, be coupled as in Section~\ref{sec:coupling}. There 
	exist $a_0,c_1,c_2>0$ such that, for all~$a,b$ with $a_0 \leq a < b \leq a 
	+ a^{2/3}$,
	\[
		\P\bigl( S_b - S_a \leq \tfrac{1}{2} (b-a) \bigr)
		\geq 1 - c_1 e^{-c_2(b-a)}.
	\]
\end{lemma}

\begin{proof}
	Let $Y_i^n$, $S_i^n$ and~$Z_i^n$ be coupled as in 
	Section~\ref{sec:coupling}. First we will show that there exists $c_3>0$ 
	such that, for all $a,b$ sufficiently large and satisfying $a < b \leq a + 
	a^{2/3}$,
	\begin{equation}
		\label{ingr2.1}
		\P\bigl( Z^b_{b-a} - a > 0.01 (b-a) \bigr)
		\leq e^{-c_3(b-a)}.
	\end{equation}
	Note that $Z^b_{b-a} - a$ is the number of times the $Z^b$~process does 
	not decrease in the first $b-a$ steps. By~\eqref{Mprocess3}, for $0\leq i 
	< b-a$, the conditional probability that $Z^b$ does not decrease in the 
	$(i+1)$-th step satisfies
	\[
		\P\bigl( Z^b_{i+1} = Z^b_i \bigm| Z^b_i \bigr)
		\leq 1 - \frac{a+1-1}{b-1} \leq 1 - \frac{a}{b}.
	\]
	Therefore, $Z^b_{b-a} - a$ is stochastically smaller than a random 
	variable~$W$ having the binomial distribution with parameters $n= b-a$ and 
	$p = 1-\tfrac{a}{b}$. If $a$ and~$b$ are such that $1-\tfrac{a}{b} < 
	0.01$, then we can use Hoeffding's inequality to bound the left hand side 
	of~\eqref{ingr2.1} by
	\[
		\P(W > 0.01 (b-a)) \leq \exp \bigl[ -2(b-a) (0.01-p)^2 \bigr].
	\]
	It follows that~\eqref{ingr2.1} holds.

	Next we will show that there exist $c_4,c_5>0$ such that, for all $a,b$ 
	sufficiently large and satisfying $a <  b \leq a + a^{2/3}$,
	\begin{equation}
		\label{ingr2.2}
		\P\bigl( Z^b_b - Y^a_a > \tfrac{1}{2} (b-a) \bigm| Z^b_{b-a} - a \leq 
		0.01 (b-a) \bigr) \leq c_4 e^{-c_5(b-a)}.
	\end{equation}
	To prove~\eqref{ingr2.2}, we consider the process of differences 
	$Z^b_{b-a+i} - Y^a_i$ at the steps~$i$ at which the $Y^a$~process 
	decreases, and bound the probability that at such steps the $Z^b$~process 
	does not decrease. Using the coupled transition probabilities 
	\eqref{Yprocess} and~\eqref{Zprocess}, for $k < a$ and $l > 0$ we have
	\begin{align}
		& \P\bigl( Z^b_{b-a+i+1} = k+l \bigm| Y^a_{i+1} = k-1, Y^a_i = k, 
		Z^b_{b-a+i} = k+l \bigr) \nonumber \\
		&\qquad\null = \biggl[ \left( \frac{k}{a-1} - \frac{k+l-1}{b-1} 
		\right) \frac{a-1}{k} \biggr]^+
		 \leq \biggl[ 1 - \frac{a-1+l-1}{b-1} \biggr]^+ \nonumber\\
		&\qquad\null \leq \biggl[ \frac{b-a - (l-2)}{b} \biggr]^+. 
		\label{ingr2.2a}
	\end{align}
	The upper bound~\eqref{ingr2.2a} also holds for $k=a$.
	
	Next we define a pure birth process~$X_t$, $t=0,1,\dotsc$ with the 
	properties that (i) $X_0 \geq \floor{0.01(b-a)}$, and (ii) when their 
	heights are the same, the birth process~$X_t$ increases with a higher 
	probability than the process of differences $Z^b_{b-a+i}-Y^a_i$. To define 
	the process~$X_t$, let
	\[
		X_0 = x_0 := \max\{ 2, \floor{0.01 (b-a)} \},
	\]
	and let the dynamics of~$X_t$ be given by
	\[\begin{split}
		\P(X_{t+1} = X_t + 1 \mid X_t = x_0 + x)
		&= 1 - \P(X_{t+1} = X_t \mid X_t = x_0 + x) \\
		&= \frac{b-a-x}{b}
	\end{split}\]
	for $x = 0,1,\dots,b-a$. Since the two processes can be coupled in such a 
	way that on the event $\{Z^b_{b-a} - a \leq 0.01(b-a)\}$, the birth 
	process~$X_t$ dominates the process of differences, we have that
	\begin{multline}
		\label{ingr2.3}
		\P\bigl( Z^b_b - Y^a_a > \tfrac{1}{2} (b-a) \bigm| Z^b_{b-a} - a \leq 
		0.01 (b-a) \bigr) \\
		\leq \P\bigl( X_{t_0} > \tfrac{1}{2} (b-a) \bigr) + \P\bigl( Y^a_a < 
		\tfrac{a}{e} - (\tfrac{a}{e})^{5/6} \bigr),
	\end{multline}
	where
	\[
		t_0 = \floor[\big]{a - \bigl( \tfrac{a}{e} - ( \tfrac{a}{e} )^{5/6} 
		\bigr)}.
	\]

	To bound the first term on the right in~\eqref{ingr2.3}, we use a 
	continuous-time version of the process~$X_t$. Let $T_x$, $x = 0, 1, \dots, 
	b-a$, be independent such that $T_x$ has the geometric distribution with 
	parameter $p = (b-a-x)/b$ (where $T_{b-a}=\infty$ with probability~1). We 
	can then write
	\[
		X_t = x_0 + \min \{ x\colon T_0 + \dots + T_x > t \},
	\]
	Now let $T_x'$, $x=0,1,\dots,b-a$, be independent such that $T_x'$ has the 
	exponential distribution with parameter
	\[
		\lambda = \log \frac{b}{a} - x \frac{\log b - \log a}{b-a}.
	\]
	We define the continuous-time process~$X_t'$, $t\geq 0$, by
	\[
		X_t' = x_0 +  \min\{ x\colon T_0'+\dots+T_x' > t \}.
	\]
	Since
	\[\begin{split}
		\P(T_x > t)
		&= \left( 1 - \frac{b-a-x}{b} \right)^{\floor{t}}
		 \geq e^{-t \bigl[ -\log\bigl( 1 - \frac{b-a-x}{b} \bigr) \bigr] } \\
		&\geq e^{-t \bigl[ \frac{b-a-x}{b} \sum_{k=1}^{\infty} \frac{1}{k} 
		(\frac{b-a}{b})^{k-1} \bigr] }
		= e^{-t \bigl[ \frac{b-a-x}{b-a} \log \frac{b}{a} \bigr] }
		= \P(T_x' > t),
	\end{split}\]
	we have that $T_x'$ is stochastically less than~$T_x$. It follows that 
	$X_t$ is stochastically dominated by~$X_t'$, hence
	\begin{equation}
		\label{ingr2.4}
		\P\bigl( X_{t_0} > \tfrac{1}{2} (b-a) \bigr)
		\leq  \P\bigl( X_{t_0}' > \tfrac{1}{2} (b-a) \bigr).
	\end{equation}

	By Lemma~\ref{lem:binomial}, $X'_{t_0} - x_0$ has the same law as a random 
	variable~$W'$ having the binomial distribution with parameters
	\[
		n' = b-a,\qquad
		p' = 1 - \exp\left( - \frac{\log b-\log a}{b-a} t_0 \right).
	\]
	We have, as $a\to\infty$,
	\[
		p'
		\leq 1 - \exp\bigl( -\tfrac{1}{a} \floor[\big]{a - \tfrac{a}{e} + 
		(\tfrac{a}{e})^{5/6}} \bigr)
		\to 1 - e^{-( 1 - e^{-1} )} \approx 0.4685.
	\]
	Therefore, if $0.01(b-a)\geq2$ and $a$ is sufficiently large, then using 
	Hoeffding's inequality we can bound $\P\bigl( X_{t_0}' > \tfrac{1}{2} 
	(b-a) \bigr)$ above by
	\[
		\P\bigl( W' > 0.49 (b-a) \bigr)
		\leq \exp \bigl[ -2(b-a) (0.49 - p')^2 \bigr]
		\leq e^{-c_6 (b-a)}
	\]
	for some constant $c_6 > 0$ that does not depend on~$a,b$. Hence
	\begin{equation}
		\label{ingr2.5}
		\P\bigl( X_{t_0}' > \tfrac{1}{2} (b-a) \bigr)
		\leq e^{200 c_6} e^{-c_6 (b-a)}
	\end{equation}
	for all values of~$b-a$ and sufficiently large~$a$.

	By Corollary~\ref{cor:tailY} and the assumption that $b -a \leq a^{2/3}$, 
	the second term on the right in~\eqref{ingr2.3} satisfies
	\begin{equation}
		\label{ingr2.6}
		\P\bigl( Y^a_a < \tfrac{a}{e} - (\tfrac{a}{e})^{5/6} \bigr)
		\leq e^{-c_7 a^{2/3}}
		\leq e^{-c_7 (b-a)},
	\end{equation}
	for some constant $c_7>0$ that does not depend on $a,b$. Combining 
	\eqref{ingr2.3}, \eqref{ingr2.4}, \eqref{ingr2.5} and~\eqref{ingr2.6} 
	gives~\eqref{ingr2.2}. Since $Y_a^a \leq S_a$ and $S_b \leq Z^b_b$, 
	\eqref{ingr2.1} and~\eqref{ingr2.2} imply the statement of the lemma.
\end{proof}

\subsection{Proof of Propositions \ref{prop:ingr3}, \ref{prop:ingr2} 
and~\ref{prop:ingr1}}
\label{sec:proofpropositions}

\begin{proof}[Proof of Proposition~\ref{prop:ingr3}.]
	The proposition is a corollary of Lemma~\ref{lem:intervals} applied for a 
	specific sequence of intervals, as we explain below. We define, for every 
	$x\geq 0$, a sequence of intervals $I_k(x)$, $k\geq 0$, as follows. Let
	\[
		\delta_x = \tfrac{1}{12} e^{-\frac{1}{3}x},
	\]
	and let $I_0(x) = [ \floor{I_0^-(x)}, \ceil{I_0^+(x)} ]$ with
	\[
		I_0^-(x) = e^{x-2\delta_x},\qquad I_0^+(x) = e^{x+2\delta_x}.
	\]
	Let $s_0$ be defined by~\eqref{s_0}, let $\gamma=1/4$ and let $c_0 = 
	c_0(x)$ be given by~\eqref{c_0}, i.e.,
	\[
		c_0(x) = \frac{1}{4s_0} e^{\frac{1}{2}x-\frac{1}{2}} (e^{\delta_x} - 
		e^{-\delta_x}).
	\]
	For $k\geq 1$, we define the interval~$I_k(x)$ by \eqref{left} 
	and~\eqref{right}, with $I_0(x)$ and $c_0=c_0(x)$ as above, i.e., $I_k(x) 
	= [ \floor{I_k^-(x)}, \ceil{I_k^+(x)} ]$ with
	\begin{align*}
		I_k^-(x)
		&= e^{x+k-2\delta_x} \Bigl( 1 + c_0(x) e^{\frac{1}{2} - \frac{1}{2}x + 
		\delta_x} \sum\nolimits_{i=1}^k \sqrt{i} \, e^{-i/2} \Bigr), \\
		I_k^+(x)
		&= e^{x+k+2\delta_x} \Bigl( 1 - c_0(x) e^{\frac{1}{2} - \frac{1}{2}x - 
		\delta_x} \sum\nolimits_{i=1}^k \sqrt{i} \, e^{-i/2} \Bigr).
	\end{align*}

	We will apply Lemma~\ref{lem:intervals} for the sequence of 
	intervals~$I_k(x)$ defined above. Since $e^{u} - e^{-u} = 2u + O(u^3)$ as 
	$u\downarrow 0$, for our choice of intervals we have
	\[
		c_0(x)
		= \frac{1}{24 s_0} e^{\frac16 x - \frac12}
			+ O\bigl( e^{-\frac12 x} \bigr)
		\to \infty \text{ as $x\to\infty$},
	\]
	hence the right hand side of~\eqref{boundintervals} tends to~0 as $x\to 
	\infty$. Now let $\varepsilon>0$ and~$a_2$ be given. By 
	Lemma~\ref{lem:intervals}, there exists $k_0 > 1+\log a_2$ such that
	\begin{equation}
		\label{ingr3.1a}
		\inf_{x\in[k_0,k_0+1]} \P\bigl( \text{for all $\textstyle n \in 
		\bigcup_{k=1}^{\infty} I_k(x)$}, X^n_{k_n} \in I_0(x) \bigr)
		\geq 1-\varepsilon.
	\end{equation}
	Choose
	\[
		\delta := \delta_{k_0+1} = \tfrac{1}{12} e^{-\frac{1}{3}(k_0+1)}.
	\]
	We will prove that for all $x\in [k_0,k_0+1]$,
	\begin{align}
		\label{ingr3.1b}
		I_k(x)
		&\supset [e^{x+k-\delta}, e^{x+k+\delta} ]
			\text{ for all $k\geq 1$}, \\
		\label{ingr3.1c}
		I_0(x)
		&\subset [e^{x-3\delta}, e^{x-3\delta} + (e^{x-3\delta})^{2/3} ].
	\end{align}
	Together, \eqref{ingr3.1a}, \eqref{ingr3.1b} and~\eqref{ingr3.1c} imply 
	the statement of the proposition, where $x$ plays the role of $k_0+w$ in 
	the proposition.

	To prove \eqref{ingr3.1b} and~\eqref{ingr3.1c}, let $x\in [k_0,k_0+1]$. 
	The inclusion~\eqref{ingr3.1b} follows from the observations that for all 
	$k\geq 1$,
	\begin{align*}
		\floor{I_k^-(x)}
		&\leq e^{x+k} \bigl( e^{-2\delta_x} + c_0(x) s_0 e^{-\frac12 x - 
		\delta_x + \frac12} \bigr)
		\leq e^{x+k} \bigl( \tfrac34 e^{-2\delta} + \tfrac14 \bigr)
		\leq e^{x+k-\delta}, \\
		\intertext{where we have used in the last two steps that $\delta \leq 
		\delta_x \leq 1/12$, and likewise}
		\ceil{I_k^+(x)}
		&\geq e^{x+k} \bigl( e^{2\delta_x\phantom+} - c_0(x) s_0 e^{-\frac12 x 
		+ \delta_x +\frac12} \bigr)
		\geq e^{x+k} \bigl( \tfrac34 e^{2\delta\phantom+} + \tfrac14 \bigr)
		\geq e^{x+k+\delta}.
	\end{align*}
	Next we prove the inclusion~\eqref{ingr3.1c}. Since $\tfrac14 e^{-1/3} > 
	\tfrac16$, for $k_0$ sufficiently large we have that
	\[
		\floor{I_0^-(x)}
		= \floor[\big]{\exp\bigl( x - \tfrac16 e^{-\frac{1}{3} x} \bigr)}
		\geq \exp \bigl( x - \tfrac14 e^{-\frac13(k_0+1)} \bigr)
		= e^{x-3\delta},
	\]
	and similarly $\ceil{I_0^+(x)} \leq e^{x+3\delta}$. Moreover, using the 
	inequalities $e^{1-6\delta} \geq 1-6\delta \geq 1-\tfrac12 e^{-x/3}$ we 
	obtain
	\[
		e^{x-6\delta} + e^{\frac23 x-5\delta}
		\geq \bigl( e^x + e^{\frac23 x} \bigr) (1-6\delta)
		\geq e^x + \tfrac12 e^{\frac23 x} - \tfrac12 e^{\frac13 x}
		\geq e^x,
	\]
	from which it follows that
	\[
		 e^{x-3\delta} + \bigl( e^{x-3\delta} \bigr)^{2/3}
		 \geq e^{x+3\delta}.
	\]
	This proves~\eqref{ingr3.1c}, and completes the proof of the 
	proposition.
\end{proof}

\begin{proof}[Proof of Proposition~\ref{prop:ingr2}.]
	The idea is to repeatedly apply Lemma~\ref{lem:ingr2} to the coupled 
	processes $X^a_i$ and~$X^b_i$, with $a < b\leq a+a^{2/3}$, until the 
	distance $X^b_i-X^a_i$ has decreased to a constant. To this end, however, 
	$X^a_i$ and~$X^b_i$ should satisfy the conditions of Lemma~\ref{lem:ingr2} 
	at each round, and this requires that we first strengthen the statement of 
	the lemma somewhat.

	Let $S_n$, $n\geq 2$, be coupled as in Section~\ref{sec:coupling}. By 
	Lemma~\ref{lem:coupling} and Corollary~\ref{cor:tailY}, and since 
	$\tfrac4{11} < e^{-1}$, there exists $c_0>0$ such that for all~$a$,
	\begin{equation}
		\label{cor2.0}
		\P\bigl( S_a \geq \tfrac{4}{11} a \bigr) \geq 1 - e^{-c_0 a}.
	\end{equation}
	Next, note that it is a deterministic fact that if $S_b-S_a \leq 
	\tfrac12(b-a)$, $b-a \leq a^{2/3}$ and $S_a \geq \tfrac{4}{11} a$, then
	\begin{equation}
		\label{cor2.01}
		S_b-S_a
		\leq \tfrac12 a^{2/3}
		\leq \tfrac12 \bigl( \tfrac{11}{4} S_a \bigr)^{2/3}
		\leq S_a^{2/3}.
	\end{equation}
	By \eqref{cor2.0}, \eqref{cor2.01} and Lemma~\ref{lem:ingr2}, there exist 
	$a_0^*,c_1,c_2 >0$ such that, for all~$a,b$ with $a_0^*\leq a < b\leq 
	a+a^{2/3}$,
	\begin{equation}
		\label{cor2.1}
		\P\bigl( S_b - S_a \leq \min\bigl\{ \tfrac12 (b-a), S_a^{2/3} \bigr\}, 
		S_a \geq \tfrac{4}{11}a \bigr)
		\geq 1 - c_1 e^{-c_2 (b-a)}.
	\end{equation}
	The additional statements that $S_b-S_a\leq S_a^{2/3}$ and $S_a \geq 
	\tfrac{4}{11}a$ in~\eqref{cor2.1} make this version of the statement of 
	Lemma~\ref{lem:ingr2} suitable for repeated application to the coupled 
	processes $X^a_i$ and~$X^b_i$.

	Let $\varepsilon>0$. Define $a_0 := (a_0^*)^{100}$ and~$d$ such that 
	$\sum_{k=d}^{\infty} c_1 \exp(-c_2 k) \leq \varepsilon$, let $a,b$ be such 
	that $a_0\leq a<b\leq a+a^{2/3}$, and let $T_0 = \inf\{ i\colon X_i^b - 
	X_i^a \leq d \}$. We claim that
	\begin{multline}
		\label{cor2.2}
		\P\bigl( \text{$-1 \leq X^b_i - X^a_i \leq d$ and $X_i^a \geq 
		a^{0.01}$ for some~$i$} \bigr) \\
		\geq \P\bigl( \text{$X_{i+1}^b - X_{i+1}^a \leq \tfrac{1}{2} (X_i^b - 
		X_i^a)$, $X_{i+1}^a \geq \tfrac{4}{11} X_i^a$ for all $i < T_0$} 
		\bigr).
	\end{multline}
	Indeed, the following three facts together imply~\eqref{cor2.2}:
	\begin{enumerate}
		\item If $X^b_{i+1} - X^a_{i+1} \leq \tfrac12 (X^b_i - X^a_i)$ for 
			every $i<0.97\log a$, then $X^b_i - X^a_i \leq d$ for some $i\leq 
			0.97\log a$ (and hence $T_0 \leq 0.97\log a$), since
			\[
				(b-a) \bigl( \tfrac{1}{2} \bigr)^{0.97\log a-1}
				\leq 2a^{2/3} \bigl( \tfrac{1}{2} \bigr)^{0.97\log a}
				\leq 2 \leq d.
			\]
		\item If $X_{i+1}^a \geq \tfrac{4}{11} X_i^a$ in every round 
			$i<0.97\log a$, then $X^a_i \geq a^{0.01}$ for all $i\leq 0.97\log 
			a$, since
			\[
				a \bigl( \tfrac{4}{11} \bigr)^{0.97 \log a}
				\geq a^{0.01}.
			\]
		\item If $X_i^b - X_i^a > 0$, then $X^b_{i+1} - X^a_{i+1} \geq -1$ 
			a.s.\ by \eqref{Xni} and Lemma~\ref{lem:coupling}.
	\end{enumerate}

	The right hand side of~\eqref{cor2.2} is at least
	\begin{multline}
		\P\bigl( X_{i+1}^b - X_{i+1}^a \leq \min\{\tfrac{1}{2} (X_i^b - 
		X_i^a),X_{i+1}^a)^{2/3}\} \\
		\text{and $X_{i+1}^a \geq \tfrac{4}{11} X_i^a$ for all $i < T_0$} 
		\bigr)
		\geq 1-\sum_{k=d}^{\infty} c_1 e^{-c_2 k}
		\geq 1-\varepsilon, \label{cor2.3}
	\end{multline}
	where the first bound on the probability follows from repeated application 
	of~\eqref{cor2.1}. Note that on the event considered in~\eqref{cor2.3} we 
	have that $X_i^a \geq a^{0.01} \geq a_0^{0.01} \geq a_0^*$ for all 
	$i<T_0$, so that we can indeed apply~\eqref{cor2.1}. The sum 
	in~\eqref{cor2.3} is over all possible values that the distance 
	$X^b_i-X^a_i$ can assume, and all larger values. The second inequality 
	in~\eqref{cor2.3} follows from the definition of~$d$. Combining 
	\eqref{cor2.2} and~\eqref{cor2.3} yields the desired result.
\end{proof}

\begin{proof}[Proof of Proposition~\ref{prop:ingr1}.]
	Let the $S_n$, $n\geq 2$, be coupled as in Section~\ref{sec:coupling}. By 
	Lemma~\ref{lem:coupling} and Corollary~\ref{cor:tailY}, and since 
	$\tfrac4{11} < e^{-1}$, there exists $c_0>0$ such that for all~$a$,
	\begin{equation}
		\label{cor1.0}
		\P\bigl( S_a \geq \tfrac{4}{11} a \bigr) \geq 1 - e^{-c_0 a}.
	\end{equation}
	By Lemma~\ref{lem:ingr2}, there exist $a_0,c_1,c_2 >0$ such that for 
	all~$a,b$ satisfying $a_0\leq a < b\leq a+a^{2/3}$,
	\begin{equation}
		\label{cor1.1}
		\P\bigl( S_b - S_a \leq \tfrac12 (b-a) \bigr)
		\geq 1 - c_1 e^{-c_2 (b-a)}.
	\end{equation}

	Let $\varepsilon>0$ and~$d$ be given. Choose $d_0\geq d$ such that
	\begin{equation}
		\label{defd0}
		e^{-d_0} \leq \frac{\varepsilon}{3}
		\quad\text{and}\quad
		d_0 c_1^{d_0} e^{14 d_0-c_2 d_0^2}
		\leq \frac{\varepsilon}{3},
	\end{equation}
	and define $T := d_0 \ceil{\exp(14d_0)}$. Now choose $a_1^*$ such that
	\begin{equation}
		\label{defa1}
		2T e^{-c_0a_1^*} \leq \frac{\varepsilon}{3}
		\quad\text{and}\quad
		a_1^* \geq \max\bigl\{ a_0,(2d_0)^{3/2},8d_0 \bigr\},
	\end{equation}
	and set $a_1 := a_1^*(11/4)^T$. Let $a,b$ be such that $a_1\leq a < b \leq 
	a+d$, and consider the coupled processes $X^a_i$ and~$X^b_i$. Define the 
	events
	\begin{align*}
		A_k &:= \bigl\{ \text{$X^a_i, X^b_i \geq a_1^*$ for all $i\leq k$} 
		\bigr\}, \\
		B_k &:= \bigl\{ \text{$\abs{X^b_i - X^a_i}\leq 2d_0$ for all $i\leq 
		k$} \bigr\}, \\
		C_k &:= \bigl\{ \text{$\abs{X^b_i - X^a_i}\neq 0$ for all $i\leq k$} 
		\bigr\}.
	\end{align*}
	Our goal is to show that~$C_T$ has small probability, which implies that 
	with high probability, $\abs{X^b_i-X^a_i} = 0$ for some~$i$. Since
	\begin{equation}
		\label{cor1.2}
		\P(C_T) \leq \P(A_T^c) + \P(A_T\cap B_T^c) + P(A_T\cap B_T\cap C_T),
	\end{equation}
	it suffices to prove that $\P(A_T^c)$, $\P(A_T\cap B_T^c)$ and $\P(A_T\cap 
	B_T\cap C_T)$ are small.
	
	We start with~$\P(A_T^c)$. Observe that it is a deterministic fact that if 
	$a\geq a_1$ and $X^a_i \geq \tfrac4{11} X^a_{i-1}$ for all $i\leq T$, then 
	by definition of~$a_1$, $X^a_i \geq a_1^*$ for all $i\leq T$. Hence, by 
	\eqref{cor1.0} and~\eqref{defa1},
	\begin{equation}
		\label{cor1.3}
		\P(A_T^c) \leq 2T e^{-c_0a_1^*} \leq \frac{\varepsilon}{3}.
	\end{equation}
	
	Next, we turn to $\P(A_T\cap B_T^c)$. Observe that we can consider the 
	absolute differences $\abs{X^b_i - X^a_i}$, $i = 0,1,\dotsc$, as a random 
	walk starting at $b-a \leq d_0$. By the definition~\eqref{Xni} of 
	the~$X^n_i$ and Lemma~\ref{lem:coupling},
	\[
		\abs{X^b_{i+1} - X^a_{i+1}} \leq \abs{X^b_i - X^a_i} + 1
		\text{ a.s.\ for all $i \geq 0$}.
	\]
	This implies that it can only be the case that $\abs{X^b_i-X^a_i} > 2d_0$ 
	for some $i\leq T$, if there exists a $k<T-d_0$ such that $d_0\leq 
	\abs{X^b_{k+j}-X^a_{k+j}} \leq d_0+j$ holds for $j = 0,1,2,\dotsc,d_0$. 
	Hence, if we introduce the notation
	\[
		D_{k,i} := A_{k+i} \cap \bigl\{ \text{$d_0\leq 
		\abs{X^b_{k+j}-X^a_{k+j}}\leq d_0+j$ for $j=0,1,\dots,i$} \bigr\},
	\]
	then we have that $\P(A_T\cap B_T^c) \leq \sum_{k<T-d_0} \P(D_{k,d_0})$, 
	which is the same as
	\begin{equation}
		\label{cor1.4}
		\P(A_T\cap B_T^c)
		\leq \sum_{k=0}^{T-d_0-1} \prod_{i=1}^{d_0} \P(D_{k,i} \mid D_{k,i-1}) 
		\cdot \P(D_{k,0}).
	\end{equation}
	
	Now, by~\eqref{defa1}, if $X_i^a,X_i^b \geq a_1^*$ and $\abs{X_i^b - 
	X_i^a} \leq 2d_0$, then it also holds that ${\abs{X^b_i - X^a_i}} \leq 
	\min\{ X^a_i,X^b_i \}^{2/3}$. Moreover, if the absolute difference 
	$\abs{X^b_i-X^a_i}$ is strictly less than~$2d_0$, it will drop below~$d_0$ 
	if it decreases by at least $\tfrac12\abs{X^b_i-X^a_i}$ at the next step. 
	Therefore, it follows from~\eqref{cor1.1} that
	\begin{equation}
		\label{cor1.5}
		\P(D_{k,i} \mid D_{k,i-1})
		\leq \P\bigl( \text{$\abs{X^b_{k+i}-X^a_{k+i}}\geq d_0$} \bigm| 
		D_{k,i-1} \bigr)
		\leq c_1 e^{-c_2 d_0}
	\end{equation}
	for $i\leq d_0$. Together, \eqref{cor1.4}, \eqref{cor1.5} 
	and~\eqref{defd0} give
	\begin{equation}
		\label{cor1.6}
		\P\bigl( A_T\cap B_T^c \bigr)
		\leq (T-d_0) \bigl( c_1e^{-c_2 d_0} \bigr)^{d_0}
		\leq d_0 c_1^{d_0} e^{14 d_0-c_2 d_0^2}
		\leq \frac{\varepsilon}{3}.
	\end{equation}

	Finally, we consider $\P(A_T\cap B_T\cap C_T)$. To simplify the notation, 
	write $E_i = A_i\cap B_i\cap C_i$ for $i\geq0$. Then we have
	\[
		\P(E_T)
		= \prod_{i=1}^T \P(E_i \mid E_{i-1}) \P(E_0)
		\leq \prod_{i=1}^T \P(C_i\mid E_{i-1}).
	\]
	Since on the event~$E_i$, $\abs{X^b_i-X^a_i} \leq 2d_0$ and $2d_0 \leq 
	\tfrac14 a_1^*\leq \tfrac14\min\{X^b_i,X^a_i\}$, by Lemma~\ref{lem:ingr1} 
	each factor in this product is bounded above by $1-e^{-14d_0}$. Hence, 
	using the inequality $1-u \leq e^{-u}$ and~\eqref{defd0},
	\begin{equation}
		\label{cor1.7}
		\P(A_T\cap B_T\cap C_T)
		= \P(E_T)
		\leq \bigl(1-e^{-14d_0}\bigr)^T
		\leq e^{-d_0}
		\leq \frac{\varepsilon}{3}.
	\end{equation}
	Combining \eqref{cor1.2}, \eqref{cor1.3}, \eqref{cor1.6}, 
	and~\eqref{cor1.7} gives
	\[
		\P\bigl( \text{$\abs{X^b_i-X^a_i}=0$ for some $i\leq T$} \bigr)
		= 1-P(C_T)
		\geq 1-\varepsilon.\qedhere
	\]
\end{proof}

\section*{}
\subsection*{Acknowledgment:}

We thank Henk Tijms for drawing our attention to the group Russian roulette 
problem.

\bibliographystyle{amsplain}
\bibliography{shooting}

\end{document}